\documentclass{article}
\usepackage[a4paper,top=3.cm,bottom=4.cm,left=3.5cm,right=3.5cm]{geometry}
\usepackage{graphicx} 
\usepackage{authblk}
\DeclareMathAlphabet{\mathpzc}{OT1}{pzc}{m}{it}

\usepackage{xcolor}
\usepackage[hidelinks]{hyperref}
\usepackage[backend=biber, useprefix=true, maxbibnames=5]{biblatex}
\usepackage{algorithm}
\usepackage{algorithmic}

\usepackage[labelfont=sc]{caption}
\usepackage{titlesec}
\titleformat{\paragraph}[runin]{\normalfont\bfseries}{\theparagraph.}
{0.5em}{}[.] 
\titleformat{\section}{\Large\bfseries}{\thesection.}
{0.5em}{}[] 
\titleformat{\subsection}{\large\bfseries}{\thesubsection.}
{0.5em}{}[] 
\usepackage[titletoc,title]{appendix}

\usepackage{amsthm}
\newtheorem{theorem}{Theorem}[section]
\newtheorem{proposition}{Proposition}[section]
\newtheorem{lemma}{Lemma}[section]
\newtheorem{remark}{Remark}[section]
\newtheorem{corollary}{Corollary}[section]

\usepackage{amsmath, amssymb, amsfonts}
\usepackage{bm}
\usepackage{mathrsfs}

\usepackage{comment}
\usepackage{enumitem}

\usepackage{xcolor}
\definecolor{darkblue}{HTML}{00008B}
\definecolor{darkred}{HTML}{8b0000}

\usepackage{array}
\usepackage{multirow}
\usepackage{booktabs}
\usepackage{siunitx}
\newcommand{\review}[1]{#1}
\newcommand{\finaltouches}[1]{#1}

\newcommand{\opnorm}[1]{{\left\vert\kern-0.25ex\left\vert\kern-0.25ex\left\vert #1
    \right\vert\kern-0.25ex\right\vert\kern-0.25ex\right\vert}}
\newcommand{\bu}{\mathbf{u}}
\newcommand{\bv}{\mathbf{v}}

\newcommand{\bz}{\mathbf{z}}
\newcommand{\bc}{\mathbf{c}}
\newcommand{\bq}{\mathbf{q}}
\newcommand{\bmu}{\boldsymbol{\mu}}
\newcommand{\bV}{\mathbf{V}}

\newcommand{\bW}{\mathbf{W}}

\newcommand{\act}{\rho}
\newcommand{\invact}{\rho^{-1}}

\newcommand{\lip}[1]{\textnormal{Lip}(#1)}
\newcommand{\ndim}{n_{0}}
\newcommand{\hypclass}{\mathcal{H}\mathcal{C}}
\newcommand{\orthparam}{\pi_{\textnormal{orth}}}

\newcommand{\bias}{\mathbf{b}}
\newcommand{\autoencoder}{\boldsymbol{\Psi}}
\newcommand{\biasE}{\mathbf{e}}
\newcommand{\biasD}{\mathbf{d}}
\newcommand{\bVE}{\mathbf{E}}
\newcommand{\bVD}{\mathbf{D}}
\newcommand{\rank}{\textnormal{rk}}
\newcommand{\cmatrix}{\boldsymbol{\Sigma}}

\renewenvironment{abstract}{\small\textsc{\abstractname.}}{\par \vspace{\baselineskip}}
\newenvironment{keywords}{\small \textsc{Keywords:}}{\par}
\newenvironment{MSCcodes}{\small \textsc{MSC:}}{\par}

\begin{document}

\title{\textsc{Deep Symmetric Autoencoders from the Eckart-Young-Schmidt Perspective\footnote{Paper accepted for publication in the \textit{SIAM Journal on Scientific Computing} (SISC) on July 7, 2026.}}}
\author[$^\dagger$]{Simone Brivio}
\author[$^\dagger$]{Nicola Rares Franco}
\affil[$^\dagger$]{\normalsize MOX, Department of Mathematics, Politecnico di Milano, Milan, Italy}

\date{}
\maketitle
\vspace{-1.25cm}
\begin{center}
{\small \{\texttt{simone.brivio}, \texttt{nicolarares.franco}\} \texttt{@polimi.it}}
\end{center}

\maketitle

\begin{abstract}
    Deep autoencoders have become a fundamental tool in various machine learning applications, ranging from dimensionality reduction and reduced order modeling of partial differential equations to anomaly detection and neural machine translation. Despite their empirical success, a solid theoretical foundation for their expressiveness remains elusive, particularly when compared to classical projection-based techniques. In this work, we aim to take a step forward in this direction by presenting a comprehensive analysis of what we refer to as symmetric autoencoders, a broad class of deep learning architectures \review{that have surfaced frequently} in the \finaltouches{recent} literature. Specifically, we introduce a formal distinction between different classes of symmetric architectures, analyzing their strengths and limitations from a mathematical perspective. For instance, we show that the reconstruction error of symmetric autoencoders with orthonormality constraints can be understood by leveraging the well-renowned Eckart-Young-Schmidt (EYS) theorem. As a byproduct of our analysis, we end up developing the EYS initialization strategy for symmetric autoencoders, which is based on an iterated application of the Singular Value Decomposition (SVD). To validate our findings, we conduct a series of numerical experiments where we benchmark our proposal against conventional deep autoencoders, discussing the importance of model design and initialization.
\end{abstract}



\section{Introduction}
Originally developed \finaltouches{between the '80s and '90s} 
as a nonlinear alternative to Principal Component Analysis (PCA) \finaltouches{\cite{rumelhart1986learning, kramer1991nonlinear, kramer1992autoassociative, hinton1993autoencoders}}, deep autoencoders have now gained traction in several fields, ranging from statistical learning to scientific machine learning. They are currently employed for tasks involving dimensionality reduction \cite{hinton2006reducing}, anomaly detection \cite{chen2018autoencoder}, image processing \cite{balle2017endtoend}, neural machine translation \cite{cho2014properties}, \review{generative modeling \cite{kingma2019vae}, } reduced order modeling of partial differential equations (PDEs) \cite{mucke2021reduced, franco2023deep}\review{,} and more. 

The widespread adoption of autoencoders can be attributed to their ability to deliver competitive performance compared to traditional linear techniques ---such as PCA, Proper Orthogonal Decomposition (POD), and the Kosambi-Karhunen-Loève (KKL) \cite{chatterjee2000introduction, loeve1978probability}--- particularly when modeling complex phenomena that exhibit nonlinear behavior. However, our theoretical understanding of deep autoencoders is still far from being exhaustive. Compared to linear methods, which are very well-understood, autoencoders remain sort of a black-box approach. \review{For this reason, } 
domain practitioners are often brought to rely upon empirical rules of thumb when it comes to their design, implementation and training. In this work, we aim to take a step forward in this direction by presenting a comprehensive analysis focused over a particular class of architectures which we refer to as \emph{symmetric autoencoders}. 
\vspace{0.15cm}

\paragraph{Existing literature} In the literature of deep autoencoders, symmetric architectures have emerged consistently during the years, with \review{different} researchers independently converging towards this design paradigm in an attempt to enhance model interpretability and reliability. For instance, in a recent work, Otto et al. \cite{otto2023learning} have proposed the use of symmetric autoencoders with biorthogonality constraints for the study of dynamical systems. Conversely, in \cite{abiri2019establishing}, a symmetric denoising autoencoder is proposed for the imputation of missing data. Autoencoder networks with a symmetric\review{-like} structure are also commonly found in the field of reduced order modeling of parametrized PDEs, see, e.g. \cite{mucke2021reduced, pichi2024graph, fresca2021comprehensive}.

Despite this convergence, however, most of these works have been carried out independently,  typically motivated by different applications and spanning across different research domains. In turn, no clear systematic investigation has been carried out yet, either theoretically or experimentally.\vspace{0.15cm}

\paragraph{Our contribution} Within this work, we provide a mathematically rigorous introduction to symmetric autoencoders. We first discuss the general idea, with a particular focus on the choice of the activation function. Then, we identify two specialized classes of architectures ---\review{symmetric} autoencoders with biorthogonality and orthogonality constraints, respectively. In doing so, we discuss the inherent mathematical properties of these architectures, outlining their strengths and limitations. In particular, by drawing a parallel with linear reduction methods, we are able to derive novel error bounds connecting symmetric autoencoders to the well-known Singular Value Decomposition (SVD). Then, by leveraging the Eckart-Young-Schmidt theorem, we design a new initialization strategy, specifically tailored for symmetric autoencoder networks, which \finaltouches{proves superior to classical initialization routines in terms of first-epoch loss, 
oftentimes benefiting training dynamics and convergence}.\vspace{0.15cm}

\paragraph{Work organization} The paper is organized as follows. First, in Section~\ref{sec:linear-dim-reduction}, we introduce the some notation and recall fundamental results in low-rank approximation theory that are essential to our analysis. 
Then, in Section~\ref{sec:symmetric}, we formally introduce and characterize deep symmetric autoencoders, along with a suitable theoretical framework to analyze their properties and capabilities, deriving the aforementioned error bounds. The practical implications of our findings, instead, are discussed in Section~\ref{sec:insights}, where we present our new  initialization strategy. Finally, in Section~\ref{sec:numerical-experiments}, we complement our analysis through an extensive set of numerical experiments, whereas, in the last section we draw some conclusions and indicate promising directions for future research.

\section{Mathematical backbone and notation}

\label{sec:linear-dim-reduction}

Before coming to our main topic,  which is the study of deep symmetric autoencoders, we provide a synthetic overview of the fundamental ingredients necessary for our analysis. We start by introducing some notation and recalling certain classical results concerning linear reduction, which will form the foundation of our approach. In doing so, we shall highlight the importance of the celebrated low-rank approximation theorem named after Eckart and Young \cite{eckart1936approximation}. Actually, hereby we present its extension to compact operators on Hilbert spaces—a result originally discovered by Schmidt back in 1907 during his studies on integral operators \cite{Schmidt1907}.

\subsection{Notation}
For clarity in our presentation, we introduce some notation concerning standard concepts in probability theory (such as covariance operators and Bochner-integrable random variables) and operator theory (including Hilbert-Schmidt operators, adjoints, and matrices).\vspace{0.15cm}

\paragraph{Probability setting} Let $(\mathscr{P},\mathcal{A},\mathbb{P})$ be a probability space, consisting of an sample space $\mathscr{P}$, a $\sigma$-field $\mathcal{A}$, and a probability \review{measure} $\mathbb{P}.$ Given a measurable map $f:\mathscr{P}\to[0,+\infty)$, we set
$\mathbb{E}[f]:=\int_{\mathscr{P}}f(\bmu)\mathbb{P}(d\bmu).$
For any given $q\in[1,+\infty)$, we denote by $L^q(\mathscr{P})$ the space of  measurable maps $g:\mathscr{P}\to\mathbb{R}$ with \review{finite $q$th-moment}, i.e., $\mathbb{E}|g|^q<+\infty.$ For all such maps, we define $\mathbb{E}[g]:=\int_{\mathscr{P}}g(\bmu)\mathbb{P}(d\bmu)$, as usual \cite{dudley2018real}.

For any $\ndim\in\mathbb{N}_+$, we extend the above notation to the case of vector-valued maps $\bu:\mathscr{P}\to\mathbb{R}^{\ndim}$ using Bochner integration. Precisely, we write $L^q(\mathscr{P};\mathbb{R}^{\ndim})$ for the space of Bochner $q$-integrable maps, i.e. the collection of all measurable maps $\bu:\mathscr{P}\to\mathbb{R}^{\ndim}$ such that $\mathbb{E}\|\bu\|^q<+\infty,$ with $\|\cdot\|$  the Euclidean norm. For all such maps, we set
$$\mathbb{E}[\bu]:=\left[\mathbb{E}[u_1],\dots,\mathbb{E}[u_{\ndim}]\right]^\top\in\mathbb{R}^{\ndim},$$
with $u_j(\bmu):=\mathbf{e}_j^\top\bu(\bmu)$ the $j$th output of $\bu$ and $\{\mathbf{e}_j\}_{j=1}^{\ndim}$ the canonical basis of $\mathbb{R}^{\ndim}.$ 

We recall that, up to quotient operations, $L^2(\mathscr{P})$ and
$L^2(\mathscr{P};\mathbb{R}^{\ndim})$ are Hilbert spaces once equipped with the inner products
$\langle f, g\rangle_{L^{2}(\mathscr{P})}:=\mathbb{E}[fg]$ and $\langle \bu,\bv\rangle_{L^{2}(\mathscr{P};\mathbb{R}^{\ndim})}:=\mathbb{E}\left[\bu^\top\bv\right],$ respectively.
Lastly, for all $\bu\in L^{2}(\mathscr{P};\mathbb{R}^{\ndim})$ we denote by
$$\cmatrix[\bu]:=\mathbb{E}\left[(\bu-\mathbb{E}[\bu]) (\bu-\mathbb{E}[\bu])^\top\right]=\left(\mathbb{E}[(u_i-\mathbb{E}[\review{u_i}])(u_j-\mathbb{E}[u_j])]\right)_{i,j=1}^{\ndim}\in\mathbb{R}^{n_0\times \ndim}$$
the covariance matrix of the random vector $\bu$.\vspace{0.15cm}

\paragraph{Hilbert-Schmidt operators} We write $\mathscr{H}(H_1,H_2)$ for the space of Hilbert-Schmidt operators from $H_1$ to $H_2$, each $H_i$ being a separable Hilbert space. We recall that the latter is defined as the collection of all continuous linear operators $T:H_1\to H_2$ such that
$\|T\|_{\textnormal{HS}}^2:=\sum_{i=1}^{+\infty}\|Te_i\|_{H_2}^2<+\infty,$
where $\{e_i\}_i$ is an orthonormal basis \finaltouches{of} $H_1$. The quantity $\|T\|_{\textnormal{HS}}$ is independent on the choice of such basis and it defines a norm over $\mathscr{H}(H_1,H_2)$ \cite{conway2007course}. For all such operators, we write
$\rank(T):=\dim(T(H_1)),$
to the denote their \textit{rank}, whereas we use $T^*$ for their adjoint operator, i.e., $T^*:H_2\to H_1$ such that $\langle Tg,v\rangle_{H_2}=\langle g,T^*v\rangle_{H_1}$
for all $g\in H_1$ and $v\in H_2$. We recall that, if $H_1$ and $H_2$ are both finite-dimensional then $T$ and $T^*$ can be represented by a matrix and its transpose.\vspace{0.15cm}

\paragraph{Matrices}
We write
$\mathscr{O}_{m,n}:=\{\mathbf{A}\in\mathbb{R}^{m\times n}\;:\;\mathbf{A}^\top\mathbf{A}=\mathbf{I}_n\}$
for the set of $m\times n$ orthonormal matrices. Given a positive-semidefinite matrix $\mathbf{M}\in\mathbb{R}^{m\times m}$, we write $\lambda_i(\mathbf{M})$ for its $i$th eigenvalue, listed in decreasing order and counting multiplicities. Finally, given any matrix $\mathbf{B}\in\mathbb{R}^{m\times n}$ we write $\|\mathbf{B}\|_{\textnormal{HS}}$ and $\|\mathbf{B}\|_{\textnormal{op}}$ for its Frobenius and operator norms, respectively.

\subsection{Fundamental results in linear reduction}

Given a random vector $\bu$ taking values in $\mathbb{R}^{\ndim}$, the purpose of linear reduction techniques is to find an approximate representation of $\bu$ by introducing a carefully chosen affine subspace $V\subset\mathbb{R}^{\ndim}$, of dimension $\dim(V)=n\ll \ndim$, designed such that the discrepancy
$\mathbb{E}\left[\inf_{\mathbf{v}\in V}\|\bu-\mathbf{v}\|^2\right]$
is minimized. Algebraically, this boils to down finding a
suitable matrix $\bV\in\mathscr{O}_{n_0,n}$ and a suitable vector $\bq\in\mathbb{R}^{n_0}$ such that the reconstruction error
\begin{equation}
    \label{eq:pod-minimization}
    \mathbb{E}\left\|\bu-\left[\bV\bV^\top(\bu-\bq)+\bq\right]\right\|^2,
\end{equation}
is as small as possible.

The main reason behind such construction is that, once $\bV$ and $\bq$ have been identified, one can use $\mathbf{c}:=\bV^\top(\bu-\bq)$, an $n$-dimensional random vector, as a proxy for $\bu$. In fact, $\bV$ and $\bq$ naturally define an encoding and a decoding strategy,
\begin{equation}
\label{eq:linear-encoding-decoding}
\mathbb{R}^{\ndim}\ni\bu\mapsto \bV^\top(\bu-\bq)\in\mathbb{R}^{n}\quad\quad\text{and}\quad\quad\mathbb{R}^{n}\ni\mathbf{c}\mapsto \bV\mathbf{c}+\bq\in\mathbb{R}^{\ndim},\end{equation}
respectively. Both maps are affine and 1-Lipschitz continuous, two properties that contribute to the stability and interpretability of linear reduction methods. 

What is most remarkable, however, is that the solution to Eq. \eqref{eq:pod-minimization} is known in closed form. In fact, the latter can be computed using an algorithm that is commonly referred to as Principal Orthogonal Decomposition (POD). At its core, POD relies on the fundamental Eckart-Young-Schmidt theorem, which we state below.

\begin{theorem}[Eckart-Young-Schmidt]
\label{theorem:EYS}
Let $(H_1,\|\cdot\|_1)$ and $(H_2,\|\cdot\|_2)$ be two separable Hilbert spaces. Let $T\in\mathscr{H}(H_1,H_2)$. There exists a vanishing nonincreasing sequence $s_1\ge s_2\ge\dots\ge0$, an orthonormal basis $\{\xi_i\}_i$ of $H_1$ and an orthonormal basis $\{v_i\}_i$ of $H_2$ such that
\begin{equation}\label{eq:svd}T=\sum_{i=1}^{\rank(T)}s_iv_i\langle\xi_i,\cdot\rangle_{H_1}.\end{equation}
Furthermore, for all $1\le n < \rank(T)$, the truncated operator $T_{n}:=\sum_{i=1}^{n}s_iv_i\langle\xi_i,\cdot\rangle_{H_1}$ is the best $n$-rank approximation of $T$ and satisfies
$$\|T-T_n\|_{\textnormal{HS}}^2=\min\left\{\|T-L\|_{\textnormal{HS}}^2\;:\;\rank(L)\le n\right\}=\sum_{i=n+1}^{\rank(T)}s_i^2,$$
Additionally, $s_i^2$ and $v_i$ are the eigenvalues and eigenvectors of $TT^*:H_2\to H_2$, whereas $s_i^2$ and $\xi_i$ are the eigenvalues and eigenvectors of $T^*T:H_1\to H_1.$
\end{theorem}
\begin{proof}
    See, e.g., \cite[Theorem VI.17]{reed1980methods}.
\end{proof}

For a given operator $T$, the scalars $\{s_i\}_i$ are typically referred to as \emph{singular values}. Similarly, $\{v_i\}_i$ and $\{\xi_i\}$ are known as left and right singular vectors, respectively. For this reason, the representation formula in Eq. \eqref{eq:svd} is often called \emph{singular values decomposition} (SVD), but is also known as \emph{low-rank decomposition}, \emph{polar form} or \emph{Schmidt decomposition}.
The connection between \eqref{eq:pod-minimization} and the Eckart-Young-Schmidt theorem lies in the following Proposition, which, in light of its simplicity, we present without proof.
\begin{proposition}
    \label{prop:isometry}
    Given any $\bu\in L^2(\mathscr{P};\mathbb{R}^{\ndim})$, define $T_\bu:L^2(\mathscr{P})\to\mathbb{R}^{\ndim}$ as
    $$T_\bu g:=\mathbb{E}[\bu g].$$
    The map $\bu\to T_\bu$ is a surjective linear isometry from $L^2(\mathscr{P};\mathbb{R}^{\ndim})$ to $\mathscr{H}(L^{2}(\mathscr{P}) \finaltouches{,} \mathbb{R}^{\ndim}).$
\end{proposition}

Essentially, Proposition~\ref{prop:isometry} can be understood as generalization of Riesz's representation theorem to vector-valued functionals, and it can be used to represent Bochner square-integrable random variables as Hilbert-Schmidt operators. By combining the above with the Eckart-Young-Schmidt theorem, it is straightforward to prove the following result, which is the cornerstone underlying the POD algorithm. Even though the proof is classical and can be found in many textbooks, see, e.g., \cite[Proposition 6.3]{quarteroni2015reduced} or \cite[Theorem 3.8]{lanthaler2022error}, the notation is often quite different. For this reason, we provide the interested reader with a revisitation of the proof in the Appendix. \review{As a side note, we emphasize that the isometry in Proposition~\ref{prop:isometry} is what motivates stating Theorem~\ref{theorem:EYS} in the general context of separable Hilbert spaces. In fact, although the ambient space, $\mathbb{R}^{n_0}$, is finite-dimensional, the introduction of a probability measure naturally implies an additional layer of complexity, which makes it necessary to discuss low-rank approximation in more general spaces.}

\begin{theorem}[Optimality of POD]
\label{theorem:pod}
Let $\bu\in L^{2}(\mathscr{P};\mathbb{R}^{\ndim})$. Set $r:=\rank(T_\bu)$ and fix any positive integer $n<r\le \review{n_0}.$  Consider a decomposition of $\cmatrix[\bu]$ in diagonal form given as
$$\cmatrix[\bu]=\mathbf{X}\mathbf{D}\mathbf{X}^\top,$$
where $\mathbf{D}:=\text{diag}(\lambda_{1}(\cmatrix[\bu]),\dots,\lambda_{r}(\cmatrix[\bu]))$ and $\mathbf{X}\in\mathscr{O}_{\ndim,r}$. 
Define
$\bV_*\in\mathbb{R}^{n_0\times n}$ by extracting the first $n$ columns of $\mathbf{X}$ and let $\bq_*:=\mathbb{E}[\bu]$.
Then,
\begin{multline}
    \mathbb{E}\|\bu-\left[\bV_*\bV_*^\top(\bu-\bq_*)+\bq_*\right]\|^2=\\=\min_{\substack{\bV\in\mathscr{O}_{\ndim,n}\\\bq\in\mathbb{R}^{\ndim}}}\mathbb{E}\left\|\bu-\left[\bV\bV^\top(\bu-\bq)+\bq\right]\right\|^2=\sum_{i=n+1}^r\lambda_i(\cmatrix[\bu]).
\end{multline}
\end{theorem}
\begin{proof}
See Appendix~\ref{appendix:proofs}.
\end{proof}

\paragraph{Summary and desiderata}
\noindent The above result, combined with the rich structure provided by linear projections, highlights the following key points:\vspace{0.2cm}
\begin{itemize}[leftmargin=0.6cm, itemsep=0.1cm]
    \item [{\em (i)}] the problem of linear reduction can be solved in closed form by the POD algorithm, which involves diagonalizing the 
    covariance matrix $\cmatrix[\review{\bu}]$;
    \item [{\em (ii)}] the accuracy of the approximation is driven by the decay of the eigenvalues of $\cmatrix[\review{\bu}]$, which ultimately relates to the Eckart-Young-Schmidt theorem;
    \item [{\em (iii)}] projection methods, including POD, are representation consistent, meaning that they map $\bu$ and its reconstruction $\tilde{\bu}:=\bV\bV^\top(\bu-\mathbf{q})+\mathbf{q}$ onto the same latent representation. Indeed, being $\mathbf{c}=\bV^\top(\bu-\mathbf{q})$, one has $$\tilde{\mathbf{c}}=\bV^\top\tilde{\bu}=\bV^\top\left[\bV\bV^\top(\bu-\mathbf{q})+\mathbf{q}-\mathbf{q}\right]=\bV^\top\bV\bV^\top(\bu-\mathbf{q})=\bV^\top(\bu-\mathbf{q})=\mathbf{c};$$
    \item [{\em (iv)}] POD, and linear projection methods in general, are grounded on a stable decoding mechanism, $\mathbf{c}\mapsto \mathbf{V}\mathbf{c}$. In fact, $$\|\bV\mathbf{c}-\bV\mathbf{c}'\|=\|\mathbf{c}-\mathbf{c}'\|,$$
    the two norms being respectively in $\mathbb{R}^{n_0}$ and $\mathbb{R}^{\review{n}}.$
\end{itemize}
\;\\Our purpose for the remainder of this work, is to bring similar insights in the context of nonlinear reduction by means of deep symmetric autoencoders.

\section{Symmetric autoencoders}
\label{sec:symmetric}

Let $\ndim > n_1\ge\dots\ge n_l>0$ be positive integers. A symmetric autoencoder with input dimension $n_0$, hidden dimensions $n_1,\dots, n_{l-1}$ and \emph{latent dimension} $n_l$, is a \review{sequence of tuples} of the form
$$\autoencoder=\{(\bVE_j, \bVD_j, \biasE_j, \biasD_j, \act)\}_{j=1}^{\finaltouches{l}},$$
where:

\begin{itemize}[leftmargin=*, itemsep=5pt, parsep=0pt, after=\vspace{5pt}, before=\vspace{1pt}]
    \item $\rho:\mathbb{R}\to\mathbb{R}$ is a given bilipschitz activation (cf. Lemma~\ref{lemma:bilipschitz}),
    \item $\bVE_j\in\mathbb{R}^{n_{j}\times n_{j-1}}$ and $\bVD_j\in\mathbb{R}^{n_{j-1}\times n_{j}}$ are suitable weight matrices, 
    \item $\biasE_j\in\mathbb{R}^{n_{j}}$ and $\biasD_j\in\mathbb{R}^{n_{j-1}}$ are given bias vectors.
\end{itemize}  

Each symmetric autoencoder defines an encoding procedure $\mathscr{E}(\autoencoder,\cdot):\mathbb{R}^{\ndim}\to\mathbb{R}^{n_l}$ via the iterative scheme
$$
\mathscr{E}(\autoencoder,\bu):=(e_l\circ\dots\circ e_1)(\bu),$$
where $e_j:\;\mathbf{h}\mapsto \act(\bVE_j\mathbf{h}+\biasE_j).$

Here, the action of $\act$ over column vectors is intended componentwise. Notice that, in order to ease notation, we do not stress the dependency of each $e_j$ on $\autoencoder$, although it is clear that $e_j=e_j(\autoencoder$).

Similarly, $\autoencoder$ also defines a decoding routine, $\mathscr{D}(\autoencoder,\cdot):\mathbb{R}^{n_l}\to\mathbb{R}^{\ndim}$, defined via
$$\mathscr{D}(\autoencoder,\bc):=(d_1\circ\dots\circ d_l)(\review{\bc}),$$
where $d_j:\mathbf{h}\mapsto\bVD_j\invact(\mathbf{h})+\biasD_j.$
Together, these are used to define the reconstruction algorithm
$$\mathscr{R}(\autoencoder,\bu):=\mathscr{D}(\autoencoder,\mathscr{E}(\autoencoder,\bu)),$$
which acts as a map $\mathscr{R}(\autoencoder,\cdot):\mathbb{R}^{\ndim}\to\mathbb{R}^{\ndim}$. We refer the reader to Fig.~\ref{fig:scheme} for a schematic representation: there, it will become apparent why some authors describe symmetric autoencoders as having a ``butterfly" architecture \cite{abiri2019establishing}.

\begin{figure}
    \centering
    \includegraphics[width=\linewidth]{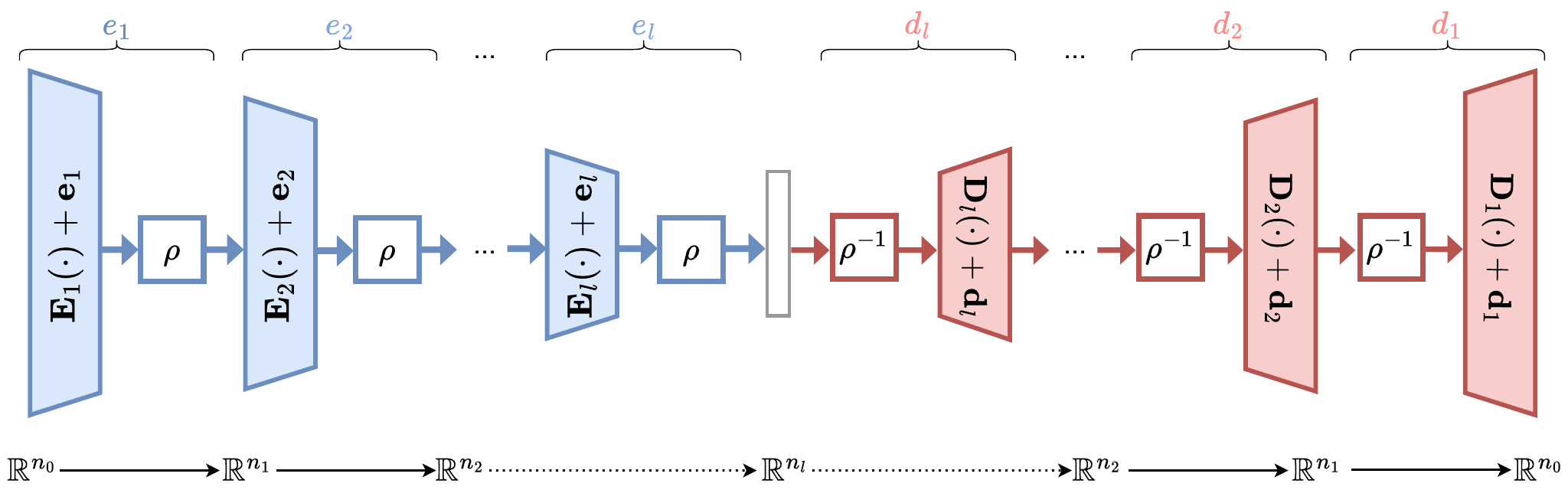}
    \caption{Schematic representation of a deep symmetric autoencoder.}
    \label{fig:scheme}
\end{figure}

\subsection{About the activation function}
We highlight the fact that, differently from other works in the literature, we require the activation function to be invertible, so that we can use $\act$ for the encoder and $\invact$ for the decoder, respectively. The same construction can be found in \cite{otto2023learning}.
We argue that this requirement is fundamental to obtain a truly symmetric architecture, as it highlights the specular role of the encoding and decoding mechanisms.

Additionally, in order to guarantee a minimal form of stability, we also require $\act$ to be bilipschitz. It can be shown, in fact, that this is equivalent to having both $\act$ and $\invact$ to be Lipschitz continuous: see Lemma~\ref{lemma:bilipschitz} right below.

\begin{lemma}
    \label{lemma:bilipschitz}
    Let $\act:\mathbb{R}\to\mathbb{R}$ be bilipschitz, meaning that there exists two positive constants, $\eta>0$ and $L>0$, such that
    \begin{equation}
        \label{eq:bilipschitz}\eta|x-y|\le|\act(x)-\act(y)|\le L|x-y|\quad\forall x,y\in\mathbb{R}.
    \end{equation}
    Then, $\act$ is bijective and its inverse is also bilipschitz. Furthermore, Eq.~\eqref{eq:bilipschitz} holds with $\eta=1/\lip{\invact}$ and $L=\lip{\act}.$
\end{lemma}
\begin{proof}
    See Appendix~\ref{appendix:proofs}.
\end{proof}

\sloppy{
This bilipschitz control will be fundamental for our construction, as it ensures a stable behavior both during training and during inference. Indeed, it is straightforward to prove that, for any $\autoencoder = \{(\bVE_j, \bVD_j, \biasE_j, \biasD_j, \act)\}_{j=1}^{\finaltouches{l}}$, one has $\lip{\mathscr{R}(\autoencoder, \cdot)} \propto (\lip{\act}\lip{\invact})^{l-1}$.} In this respect, we take the chance to introduce the notion of \emph{sharpness}, a quantity that is specific of bilipschitz activations and that we define as
\begin{equation}
    \label{eq:sharpness}
    \angle(\act) = \lip{\act}\lip{\invact} - 1 \ge 0.
\end{equation}
Essentially, the sharpness provides a quantitative measure of the instability brought by $\act$, and also serves to quantify its degree of nonlinearity: notice, in fact, that $\angle(\act)\equiv0 $ whenever $\act$ is linear. 

Examples of bilipschitz activations include:\\

\begin{itemize}[itemsep = 4pt, leftmargin = 0.6cm]
    \item[{\em (i)}] the $(\alpha,\beta)$-LeakyReLU activation, already considered in \cite{teng2019invertible}. Given $\alpha,\beta>0$, with $\alpha\neq\beta$, the latter is defined as
    \begin{equation*}
        \operatorname{LeakyReLU}_{\alpha,\beta}(x) = \alpha x \mathbf{1}_{(-\infty,0)}(x) + \beta x \mathbf{1}_{[0,\infty)}(x),
    \end{equation*}
    where $\mathbf{1}_E$ denotes the characteristic function of the set $E$;
    
    \item[{\em (ii)}] the $\theta$-Hyperbolic activation, first proposed by Otto et al. \cite{otto2023learning}. Given $\theta\in(0,\pi/4)$, the latter reads
    \begin{equation*}
        \operatorname{HypAct}_{\theta}(x) = \frac{b_\theta}{a_\theta}x- \frac{\sqrt 2}{a_\theta \sin\theta}+ \frac{1}{a_\theta}\sqrt{\biggl(\frac{2x}{\sin\theta\cos\theta} - \frac{\sqrt 2}{\cos\theta}\biggr)^2 + 2a_\theta},
    \end{equation*}
    where $a_\theta = \csc^2(\theta) - \sec^2(\theta)$ and $b_\theta = \csc^2(\theta) + \sec^2(\theta).$\\
\end{itemize}

\begin{figure}
    \centering
    \includegraphics[width=\linewidth]{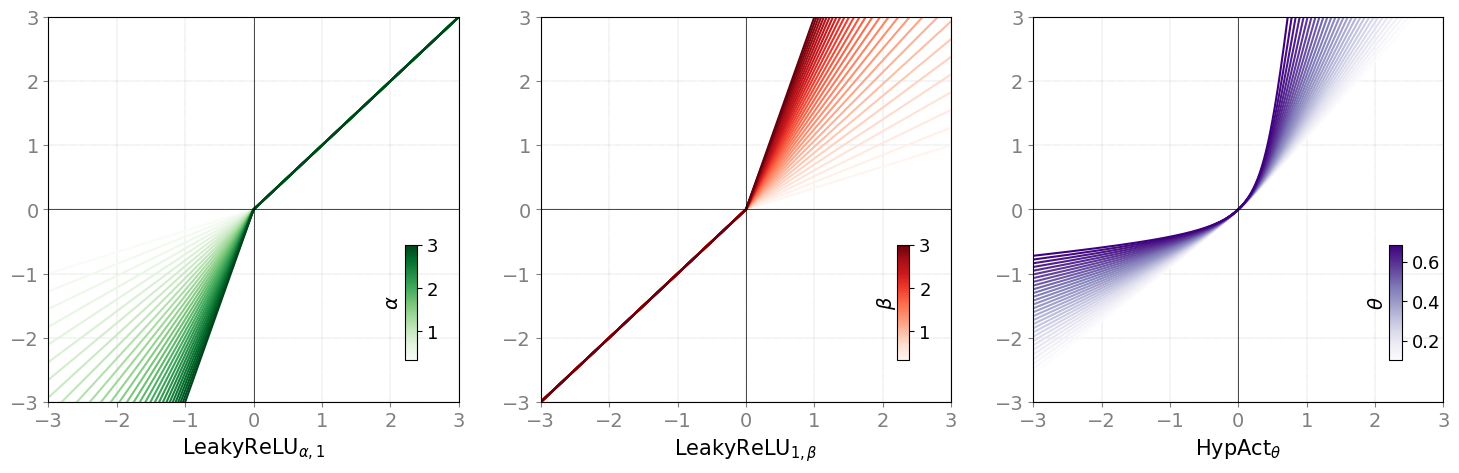}\vspace{-12pt}
    \caption{Visualization of the $(\alpha,\beta)$-LeakyReLU and $\theta$-Hyperbolic activations for different values of their hyperparameters.}
    \label{fig:activations}
\end{figure}

We refer the reader to Fig.~\ref{fig:activations} for a visualization of how the hyperparameters $\alpha,\beta$ and $\theta$ affect the two activations.
It is worth mentioning that each activation function comes with its own strenghts and limitations. The $(\alpha,\beta)$-LeakyReLU activation, for instance, has the advantage of providing a direct control over the Lipschitz constant of the nonlinearity used for encoding and decoding, respectively. In fact, if $\act=\operatorname{LeakyReLU}_{\alpha,\beta}$, then $\lip{\act}=\max\{\alpha,\beta\}$ and $\lip{\invact}=\max\{1/\alpha,1/\beta\}$. This nonlinearity, however, is not as smooth as the $\theta$-Hyperbolic activation, which, in contrast, has the benefit of being infinitely differentiable. In turn, the latter does not distinguish as much between encoding and decoding, since
$\act = \operatorname{HypAct}_{\theta}\;\implies\;\lip{\act}=\lip{\invact}=\tan\left(\theta+\frac{\pi}{4}\right).$

\newcommand{\SAE}{\textnormal{\textsf{SAE}}_\act(n_0,\dots,n_l)}
\newcommand{\SBAE}{\textnormal{\textsf{SBAE}}_\act(n_0,\dots,n_l)}
\newcommand{\SOAE}{\textnormal{\textsf{SOAE}}_\act(n_0,\dots,n_l)}

\newcommand{\SAEz}{\textnormal{\textsf{SAE}}_\act}
\newcommand{\SBAEz}{\textnormal{\textsf{SBAE}}_\act}
\newcommand{\SOAEz}{\textnormal{\textsf{SOAE}}_\act}

\subsection{Classes of symmetric autoencoders} Given a decreasing sequence of integers, $\ndim > n_1\ge\dots\ge n_l>0$, and a bilipschitz map $\act:\mathbb{R}\to\mathbb{R}$, we shall write
$$\SAE=\prod_{j=1}^{\finaltouches{l}}\left[\mathbb{R}^{n_j\times n_{j-1}}\times \mathbb{R}^{n_{j-1}\times n_{j}} \times \mathbb{R}^{n_j}\times \mathbb{R}^{n_{j-1}}\times\{\act\}\right],$$
for the space of all symmetric autoencoders associated with the given architecture.

We also introduce two additional subclasses, obtained by requiring the above networks to satisfy additional constraints related to their weights and biases. Specifically, we consider the subclass of \emph{symmetric biorthogonal autoencoders}, already introduced in \cite{otto2023learning},
\begin{align*}
    \SBAE:=\{\autoencoder\in\SAE\;:&\;\bVE_j\bVD_j=\mathbf{I}_{n_j}\;\;\text{and}\\&\;\bVE_j\biasD_j=-\biasE_j\;\;\forall j=1,\dots,\finaltouches{l}\},
\end{align*}
and the one of \emph{symmetric orthogonal autoencoders},
$$\SOAE:=\{\autoencoder\in\SBAE\;:\;\bVE_j=\bVD_j^\top\}.$$
By definition,
$\SOAE\subseteq \SBAE\subseteq \SAE.$ Notice that, as one moves from the broader to the tighter class of autoencoders, the number of architecture parameters decreases. For instance, in the biorthogonal case, the encoding bias vectors $\biasE_j$ are univocally determined by $\bVE_j$ and $\biasD_j$ \review{(observe that the latter relationship is inspired by linear methods, cf. Eq.\eqref{eq:linear-encoding-decoding}, where a decoding shift equal to $\mathbf{q}$ results in an encoding shift of $-\mathbf{V}^\top\mathbf{q}$)}. Similarly, in the orthogonal case, knowing the encoding matrix is enough to derive the decoding matrix and vice versa. 
In light of this, we shall use the following short-hand notation:
\begin{itemize}[itemsep=4pt, before=\vspace{2pt}, after=\vspace{2pt}]
    \item if $\autoencoder\in\SBAE$ we shall write $\{(\bVE_j,\bVD_j,\bias_j,\act)\}_{j=1}^{\finaltouches{l}}$ to intend the symmetric autoencoder $\{(\bVE_j,\bVD_j,-\bVE_j\bias_j,\bias_j,\act)\}_{j=1}^{\finaltouches{l}}$;
    \item if $\autoencoder\in\SOAE$ we shall write $\{(\bV_j,\bias_j,\act)\}_{j=1}^{\finaltouches{l}}$ to intend the symmetric autoencoder $\{(\bV_j^\top,\bV_j,-\bV_j^\top\bias_j,\bias_j,\act)\}_{j=1}^{\finaltouches{l}}$;
\end{itemize}
With this notation, the encoding layers read $e_j:\mathbf{h}\mapsto\act(\bVE_j(\mathbf{h}-\review{\bias_j}))$ in the biorthogonal case, and $e_j:\mathbf{h}\mapsto\act(\bV_j^\top(\mathbf{h}-\review{\bias_j}))$ in the orthogonal case. Notice, in particular, how the encoding formula used in the orthogonal case finds a natural connection to the one adopted in the context of linear reduction: cf. Eq. \eqref{eq:linear-encoding-decoding}. Similar considerations hold for the decoding phase as well. 
\vspace{0.15cm}

\sloppy{
Notwithstanding, it is worth emphasizing that the constraints fulfilled by $\SBAE$ and $\SOAE$, respectively,  give the two classes several remarkable properties. In particular, one can show that all biorthogonal symmetric autoencoders are representation consistent: a given object and its reconstruction both map onto the same latent representation. As an immediate  consequence, all $\autoencoder\in\SBAE$ act as nonlinear projectors.
Orthogonal symmetric autoencoders, in turn, have the additional benefit of providing a direct control over the Lipschitz constants of the encoder and the decoder modules, which can be extremely useful when model stability becomes a necessity. These facts are made rigorous in the Proposition below.
}
\begin{proposition}
    \label{prop:classes}
    Let $\ndim > n_1\ge\dots\ge n_l$ be positive integers and let $\act:\mathbb{R}\to\mathbb{R}$ be bilipschitz. The following hold true.

    \begin{itemize}[before=\vspace{5pt}, itemsep = 2pt, leftmargin=18pt]
        \item[(i)] $\autoencoder\in\SBAE\Rightarrow \mathscr{E}(\autoencoder,\mathscr{D}(\autoencoder,\mathbf{c}))=\mathbf{c}$ for all $\mathbf{c}\in\mathbb{R}^{n_l}.$

        \item[(ii)] $\autoencoder\in\SBAE\Rightarrow \mathscr{R}(\autoencoder,\mathbf{v})=\mathbf{v}$ for all $\mathbf{v}\in\mathscr{R}(\autoencoder,\mathbb{R}^{n_0}).$

        \item[(iii)] $\autoencoder\in\SOAE \Rightarrow\lip{\mathscr{E}(\autoencoder, \cdot)}\le \lip{\act}^l,\;\;\lip{\mathscr{D}(\autoencoder, \cdot)}\le \lip{\invact}^l$.
    \end{itemize}
\end{proposition}

\begin{proof}
    See Appendix~\ref{appendix:proofs}.
\end{proof}

\subsection{Training of symmetric autoencoders}
\label{subsec:training}
Given some $\bu\in L^2(\mathscr{P};\mathbb{R}^{n_0})$, and a corresponding training set $\bu_1,\dots,\bu_{S}\in\mathbb{R}^{n_0}$, consisting of $S$ independent random realizations of $\bu$, the practical implementation and training of a symmetric autoencoder consists in: \emph{(i)} choosing a suitable architecture, thus identifying a proper hypothesis class $\hypclass\subseteq\SAE$; \emph{(ii)} training the model by minimizing the empirical loss
\vspace{-0.2cm}
\begin{equation}
\label{eq:loss}
\mathscr{L}(\boldsymbol{\Psi})=\frac{1}{S}\sum_{i=1}^{S}\|\bu_i-\mathscr{R}(\boldsymbol{\Psi}, \bu_i)\|^2.
\vspace{-0.2cm}\end{equation}
In the case of unconstrained symmetric autoencoders, $\hypclass=\SAE$, such minimization problem can be addressed using classical optimization routines. Conversely, things become slightly more involved if we consider architectures with (bi-)orthogonality constraints, namely, if either $\hypclass=\SBAE$ or $\hypclass=\SOAE$. In those cases, in fact, one must guarantee that the optimization procedure does not exit the hypothesis manifold $\hypclass$, meaning the autoencoder must satisfy the (bi-)orthogonality constraints throughout the entire training phase.

One way to overcome this difficulty, and to make ensure that all orthogonality constrains are fulfilled \emph{by design}, is to introduce a suitable parametrization. To this end, let
$\orthparam: \mathbf{A} \mapsto \tilde{\mathbf{A}},$
be any (differentiable) algorithm that, given a rectangular matrix $\mathbf{A}\in\mathbb{R}^{m\times n}$ with $m\ge n$, returns an orthonormal matrix  $\tilde{\mathbf{A}}\in\mathscr{O}_{m,n}$ such that $\text{span}(\mathbf{A})\subseteq\text{span}(\tilde{\mathbf{A}})$, with equality when $\mathbf{A}$ has full rank. An example of such algorithm in the \texttt{pytorch} library \cite{pytorch} \finaltouches{is the one} based on the use of Householder reflections. 
\vspace{0.35cm}

\paragraph{SOAE}
Given $\act$ and $n_0,\dots,n_l$, let $\Theta:=\prod_{j=1}^{l}\left[\mathbb{R}^{n_{j-1}\times n_j}\times \mathbb{R}^{n_{j-1}}\right]$ denote our parametrization space. For any $\boldsymbol{\theta}=\{(\mathbf{A}_j,\mathbf{b}_j)\}_{j=1}^l\in\Theta$, we denote by $\boldsymbol{\Psi}_{\boldsymbol{\theta}}$ the symmetric orthogonal autoencoder obtained by setting $\boldsymbol{\Psi}_{\boldsymbol{\theta}}:=\{(\orthparam(\mathbf{A}_j),\mathbf{b}_j,\act)\}_{j=1}^l.$ It is straightforward to see that such parametrization is surjective, meaning that any $\boldsymbol{\Psi}\in\SOAE$ can be written as $\boldsymbol{\Psi}=\boldsymbol{\Psi}_{\boldsymbol{\theta}}$ for some $\boldsymbol{\theta}\in\Theta$. This allows us to reframe the learning problem in terms of the underlying parametrization, namely, using
$$\tilde{\mathscr{L}}(\boldsymbol{\theta}):=\mathscr{L}(\boldsymbol{\Psi}_{\boldsymbol{\theta}}).$$
The training phase will then consist in finding the optimal $\boldsymbol{\theta}_*\in\Theta$ that minimizes $\tilde{\mathscr{L}}$, and then assembling the corresponding autoencoder.
\vspace{0.35cm}

\paragraph{SBAE} A similar procedure can be carried out in the biorthogonal case, althought the parametrization is significantly more involved. For $q\ge r$, one can construct two matrices $\bVE\in\mathbb{R}^{r\times q}$ and $\bVD\in\mathbb{R}^{q\times r}$ satisfying $\mathbf{E}\mathbf{D}=\mathbf{I}_{r}$ by letting
\begin{equation}
\label{eq:biorthogonal-param}
\bVE^\top:=\mathbf{X}\left[\begin{array}{c}
     \mathbf{Y}\mathbf{S}\mathbf{Z}^\top\\
     \mathbf{0}
\end{array}\right]\quad\quad\text{and}\quad\quad\bVD:=\mathbf{X}\left[\begin{array}{c}
     \mathbf{Y}\mathbf{S}^{-1}\mathbf{Z}^\top\\
     \mathbf{Q}
\end{array}\right],
\end{equation}
where, having set $d:=\min\{r, q-r\}$, the matrices $\mathbf{X}\in\mathscr{O}_{q,r+d}$ and $\mathbf{Y},\mathbf{Z}\in\mathscr{O}_{r,r}$ are orthonormal, whereas $\mathbf{0}\in\mathbb{R}^{d\times r}$ is the zero matrix, $\mathbf{Q}\in\mathbb{R}^{d\times r}$ is arbitrary and $\mathbf{S}\in\mathbb{R}^{r\times r}$ is a diagonal matrix with strictly positive entries over the diagonal. Notably, this construction is surjective, in \review{the sense} that one can easily show that any two biorthogonal matrices can be written as in \eqref{eq:biorthogonal-param}: \review{see Appendix~\ref{appendix:biorth}.} Additionally, compared to other parametrizations seen in the literature, see, e.g., \cite{otto2023learning}, the one proposed here is fairly inexpensive as it only requires inverting a diagonal matrix with full rank.

In light of this fact, given $\act$ and $n_0,\dots,n_l$, we define the optimization space underlying a symmetric biorthogonal autoencoder architecture as $$\Theta:=\prod_{j=1}^{l}\left[\mathbb{R}^{n_{j-1}\times(n_j+d_j)}\times 
\mathbb{R}^{n_{j}\times n_j}\times \mathbb{R}^{n_{j}\times n_j}\times
\mathbb{R}^{d_{j}\times n_j}\times\mathbb{R}^{n_j}\times\mathbb{R}^{n_{j-1}}\right],$$
where $d_j:=\min\{n_j, n_{j-1}-n_j\}$. For any $\boldsymbol{\theta}=\{(\tilde{\mathbf{X}}_j,\tilde{\mathbf{Y}}_j, \tilde{\mathbf{Z}}_j, \mathbf{Q}_j,\mathbf{s}_j,\mathbf{b}_j)\}_{j=1}^{l}\in\Theta$ the corresponding autoencoder $\boldsymbol{\Psi}_{\boldsymbol{\theta}}\in\SBAE$ can be constructed as
$$\boldsymbol{\Psi}_{\boldsymbol{\theta}}=\{(\bVE_j,\bVD_j,\bias_j,\act)\}_{j=1}^l,$$
where each pair $(\bVE_j,\bVD_j)$ is assembled via \eqref{eq:biorthogonal-param} using the matrices $\mathbf{X}_j:=\orthparam(\tilde{\mathbf{X}}_j)$, $\mathbf{Y}_j:=\orthparam(\tilde{\mathbf{Y}}_j)$, $\mathbf{Z}_j:=\orthparam(\tilde{\mathbf{Z}}_j)$, $\mathbf{Q}_j$ and \review{$\mathbf{S}_j:=\text{diag}(\tau(\mathbf{s}_j))$, with $\tau(s):=\log_2(1+2^s)$ the softplus transformation, here applied 
componentwise on vector entries}. Now, as in the orthogonal case, the optimization can be carried out in the unconstrained optimization space $\Theta$ rather than in $\SBAE$.

\begin{remark}
    Both in the case of orthogonal and biorthogonal constraints, the proposed parametrizations are surjective, thus ensuring a proper exploration of the hypothesis class during the training phase. Additionally, this approach guarantees that the constraints are satisfied \emph{by design}. This is in contrast to other strategies adopted in the literature, where the constraints are only \emph{weakly} imposed. For example, in \cite{otto2023learning}, the authors employ biorthogonal architectures, but the training is conducted over a broader hypothesis space, $\SBAE \subset \hypclass \subset \SAE$, with the biorthogonality requirement only being enforced indirectly by including an additional penalty term in the loss function \eqref{eq:loss}.  
\end{remark}

\section{Insights from the Eckart-Young-Schmidt perspective}
\label{sec:insights}
In this Section we shall derive some quantitative error bounds on the reconstruction error of symmetric autoencoders and use those to develop a new initialization strategy. We start by framing everything in an abstract setting, but later transition to a more practical perspective, focusing on the empirical case in which the autoencoder has to be constructed based on a finite amount of data.

\subsection{Theoretical considerations}
\label{subsec:theory}
Let $\bu\in L^{2}(\mathscr{P};\mathbb{R}^{n_0})$. Given $0<n_l\le\dots\le n_1\le n_0$ and \finaltouches{a bilipschitz }$\act:\mathbb{R}\to\mathbb{R}$, we are interested in studying the reconstruction error
\begin{equation}
    \label{eq:reconstruction-error}
    \mathbb{E}\|\bu-\mathscr{R}(\autoencoder,\bu)\|^2,
\end{equation}
where $\autoencoder\in\SAE$. To start, we notice that one has the following chain of inequalities,
\begin{align}
\label{eq:ineq-chain}
\inf_{\autoencoder\in\SOAEz}\;
\mathbb{E}\|\bu-\mathscr{R}(\autoencoder,\bu)\|^2&\ge
\inf_{\autoencoder\in\SBAEz}\;
\mathbb{E}\|\bu-\mathscr{R}(\autoencoder,\bu)\|^2\\\nonumber&\ge
\inf_{\autoencoder\in\SAEz}\;
\mathbb{E}\|\bu-\mathscr{R}(\autoencoder,\bu)\|^2\ge \sum_{i=n_1+1}^{n_0}\lambda_i(\cmatrix[\bu]),
\end{align}
where we have dropped the dependency on $(n_0,\dots,n_l)$ to enhance readability.
The first two inequalities are obvious, as they stem from the fact that the autoencoder classes $\SAEz$, $\SBAEz$ and $\SOAEz$ are nested. The last inequality, instead, follows trivially by noting that for each fixed $\autoencoder=\{(\bVE_j,\bVD_j,\biasE_j,\biasD_j,\act)\}_{j=1}^{l}\in\SAE$, \review{regardless of the nonlinearites in the autoencoding process,} the outputs of $\mathscr{R}(\autoencoder,\cdot)$ \review{will always fall inside} 
the $n_1$-dimensional affine subspace \review{spanned by the last layer of the decoder module,}
$$V_{\autoencoder}:=\{\bVD_1\mathbf{c}+\biasD_1\;:\;\mathbf{c}\in\mathbb{R}^{n_1}\}\subset\mathbb{R}^{n_0}.$$
Then, as we discussed in Section~\ref{sec:linear-dim-reduction}, this immediately implies 
$\mathbb{E}\|\bu-\mathscr{R}(\Psi,\bu)\|^2\ge\mathbb{E}\left[\inf_{\mathbf{v}\in V_{\autoencoder}}\|\bu-\mathbf{v}\|^2\right]\ge \sum_{i=n_1+1}^{n_0}\lambda_i(\cmatrix[\bu])$,  \review{thus yielding} \eqref{eq:ineq-chain}.

This highlights the fact that the first hidden dimension, $n_1$, plays a crucial role, as it directly affects the overall expressivity of the architecture. Since \eqref{eq:ineq-chain} already provides a lower-bound to \eqref{eq:reconstruction-error}, our purpose now is to gain some additional insights by deriving a suitable upper-bound. To this end, leveraging \eqref{eq:ineq-chain}, we shall focus for a moment on the smaller class of symmetric orthogonal autoencoders.

In this respect, we start by stating a very useful lemma, upon which the rest of our derivation is based.

\begin{lemma}
\label{lemma:deep}
Let $\boldsymbol{\Psi}=\{(\bV_j,\bias_j,\act)\}_{j=1}^{\finaltouches{l}}\in \SOAE$, where
$\ndim\ge n_{1}\ge n_2\ge\dots\ge n_{l}$ for some $l\ge2$, and $\act:\mathbb{R}\to\mathbb{R}$ is bilipschitz. Let $e_j=e_j(\autoencoder)$ and $d_j:=d_j(\autoencoder)$. Define $E_k := e_k \circ \ldots \circ e_1$ and $D_{k} := d_k \circ \ldots \circ d_l$. Let $E_0:=\mathsf{Id}_{\ndim}$ and $D_{l+1}:=\review{\mathsf{Id}_{n_0}}$ act as the identity. Then, for all $k=1,\ldots,l$, one has
\begin{align}
    \| E_{k-1}(\bv) - D_{k}(E_l(\bv))\|^2 \ge &\;\| E_{k-1}(\bv) - [\bV_{k} \bV_{k}^\top (E_{k-1}(\bv)-\bias_k)+\bias_k]\|^2\\&+\lip{\act}^{-2} \| E_{k}(\bv) - D_{k+1}(E_l(\bv))\|^2,\nonumber\\
    \| E_{k-1}(\bv) - D_{k}(E_l(\bv))\|^2 \le &\;\| E_{k-1}(\bv) - [\bV_{k} \bV_{k}^\top (E_{k-1}(\bv)-\bias_k)+\bias_k]\|^2\\&+\lip{\invact}^{2} \| E_{k}(\bv) - D_{k+1}(E_l(\bv))\|^2,\nonumber
\end{align}
for all $\bv\in\mathbb{R}^{\ndim}$.
\end{lemma}

\begin{proof}
    Notice that for all $\bv\in\mathbb{R}^{\ndim}$
\begin{align}
\label{eq:addenda-proof-lemma}
\nonumber
\| E_{k-1}(\bv) - D_{k}(E_l(\bv))\| \overset{def}{=} &\| E_{k-1}(\bv) - d_{k} (D_{k+1}(E_l(\bv)))\|^2 \\\nonumber
\overset{def}{=} &\| E_{k-1}(\bv) - \bV_{k}  \act^{-1} (D_{k+1}(E_l(\bv)))-\bias_k\|^2 \\\nonumber
\overset{orth}{=} &\| [E_{k-1}(\bv)-\bias_k] - \bV_{k}  \bV_{k}^{\top}[E_{k-1}(\bv)-\bias_k]\|^2 \\
&\hspace{0.cm} + \| \bV_{k}  \bV_{k}^{\top} [E_{k-1}(\bv)-\bias_k] - \bV_{k}  \act^{-1} (D_{k+1}(E_l(\bv)))\|^2,      
\end{align}
the first two equalities following by definition and the last one by orthogonality. Notice, in fact, that both $\bV_{k}  \bV_{k}^{\top}[E_{k-1}(\bv)-\bias_k]$ and $\bV_{k}  \act^{-1} (D_{k+1}(E_l(\bv)))$ belong to $\text{span}(\bV_k)$, to which $[E_{k-1}(\bv)-\bias_k] - \bV_{k}  \bV_{k}^{\top}[E_{k-1}(\bv)-\bias_k]$, the projection residual, is orthogonal. We now observe that the second term at the end of Eq. \eqref{eq:addenda-proof-lemma} can be re-written as
\begin{equation*}
\begin{aligned}
    &\| \bV_{k}  \bV_{k}^{\top} [E_{k-1}(\bv) -\bias_k]- \bV_{k}  \act^{-1}(D_{k+1}(E_l(\bv)))\|^2  = \\
    &\hspace{2cm} =\| \bV_{k}^{\top}[E_{k-1}(\bv)-\bias_k] -   \act^{-1}(D_{k+1}(E_l(\bv)))\|^2 \\
    & \hspace{2cm} = \| \invact(\act(\bV_{k}^{\top}[E_{k-1}(\bv)-\bias_k])) -   \act^{-1}(D_{k+1}(E_l(\bv)))\|^2\\
    & \hspace{2cm} = \| \invact (E_{k}(\bv)) - \act^{-1}(D_{k+1}(E_l(\bv)))\|^2.
\end{aligned}
\end{equation*}
Finally, thanks to the bilipschitzness of $\act$, cf. Lemma~\ref{lemma:bilipschitz}, we obtain the chain of inequalities
\begin{equation*}
\begin{aligned}
    \lip{\act}^{-2}\| E_{k}(\bv) - D_{k+1}(E_l(\bv))\|^2 &\le \| \invact (E_{k}(\bv)) - \act^{-1}(D_{k+1}(E_l(\bv)))\|^2 \\ &\le \lip{\invact}^2 \| E_{k}(\bv) - D_{k+1}(E_l(\bv))\|^2.
\end{aligned}
\end{equation*}
The above, combined with \eqref{eq:addenda-proof-lemma}, yields the desired bounds.
\end{proof}

Lemma~\ref{lemma:deep} takes advantage of the compositional structure of deep symmetric orthogonal autoencoders to derive suitable error bounds characterizing the projection error at each level $k=1,\dots,l$.
Leveraging Lemma~\ref{lemma:deep} in an iterative fashion, we can then easily prove the following general result. 

\begin{theorem}
\label{theorem:soae}
Let $\bu\in L^{2}(\mathscr{P};\mathbb{R}^{\ndim}).$
Let $\ndim\ge n_{1}\ge n_2\ge\dots\ge n_{l}$, $l\ge2,$ and let $\act:\mathbb{R}\to\mathbb{R}$ be bilipschitz. For all $\autoencoder=\{(\bV_j,\bias_j,\act)\}_{j=1}^{\finaltouches{l}}\in\SOAE$ one has{\small\begin{equation}
\label{eq:soae-lower}
\mathbb{E}\|\bu-\mathscr{R}(\boldsymbol{\Psi}, \bu)\|^2\ge\sum_{k=0}^{l-1} \lip{\act}^{-2k}  \mathbb{E}\| E_k(\bu) - [\bV_{k+1} \bV_{k+1}^\top (E_k(\bu)-\bias_{k+1})+\bias_{k+1}]\|^2,\vspace{-0.5cm}
\end{equation}}
{\small\begin{equation}
\label{eq:soae-upper}
\mathbb{E}\|\bu-\mathscr{R}(\boldsymbol{\Psi}, \bu)\|^2\le\sum_{k=0}^{l-1} \lip{\invact}^{2k}  \mathbb{E}\| E_k(\bu) - [\bV_{k+1} \bV_{k+1}^\top (E_k(\bu)-\bias_{k+1})+\bias_{k+1}]\|^2,
\end{equation}}
where $E_k:=(e_k\circ\dots\circ e_1)$, being $e_j:=e_j(\autoencoder)$ and $E_0:=\mathsf{Id}_{\ndim}$. 
\end{theorem}
\begin{proof}
We start by proving the lower bound \eqref{eq:soae-lower}. By Lemma~\ref{lemma:deep},
\begin{multline*}
   \mathbb{E}\|\bu-\mathscr{R}(\boldsymbol{\Psi}, \bu)\|^2 = \mathbb{E}\|E_0(\bu)-D_1(E_l(\bu))\|^2 \ge \mathbb{E}\| E_0(\bu) - [\bV_{1} \bV_{1}^\top (E_0(\bu)-\bias_1)+\bias_1]\|^2\\+\lip{\act}^{-2}\mathbb{E}\|E_1(\bu)-D_2(E_l(\bu))\|^2. 
\end{multline*}
Then, iteratively applying Lemma~\ref{lemma:deep} to the last term at the right-hand-side yields
\begin{align*}
   \mathbb{E}\|\bu-\mathscr{R}(\boldsymbol{\Psi}, \bu)\|^2&\ge\dots\\\dots&\ge\sum_{k=0}^{l-1} \lip{\act}^{-2k} \mathbb{E}\| E_k(\bu) - [\bV_{k+1} \bV_{k+1}^\top (E_k(\bu)-\bias_{k+1})+\bias_{k+1}]\|^2\\&\hspace{1.5cm}+\lip{\act}^{-2l}\mathbb{E}\|E_l(\bu)-D_{l+1}(E_l(\bu))\|^2.
\end{align*}
Notice, however, that the last term is actually zero as $D_{l+1}$ is the identity. This proves \eqref{eq:soae-lower}. Mirroring the above steps, but using the upper bound result in Lemma~\ref{lemma:deep}, yields \eqref{eq:soae-upper} and thus concludes the proof.
\end{proof}

What is most interesting about Theorem~\ref{theorem:soae} is that it bounds the reconstruction error in terms of $l$ terms, one per each layer, with the $k$th 
term connecting the $k$th hidden representation in the architecture, $E_{k}(\bu)$, with the trainable parameters at level $k+1$, $\bV_{k+1}$ and $\bias_{k+1}$. Furthermore, each term in Eqs. \eqref{eq:soae-lower}-\eqref{eq:soae-upper} can be interpreted as a linear reduction, happening at level $k$.

In particular, this suggests minimizing the upper-bound in \eqref{eq:soae-upper} by addressing each term alone, following a sort of greedy approach. These considerations lead to the following Corollary.

\newcommand{\eysbar}[1]{\tilde{#1}}

\begin{corollary}
\label{corollary:loose-bounds}
Let $\bu\in L^{2}(\mathscr{P};\mathbb{R}^{\ndim})$, $\ndim\ge n_{1}\ge n_2\ge\dots\ge n_{l}$ with $l\ge2,$ and let $\act:\mathbb{R}\to\mathbb{R}$ be bilipschitz. Set 
\review{$\eysbar{E}_0:=\mathsf{Id}_{n_0}.$} Starting with $k=0$, 

\begin{itemize}[itemsep = 4pt, before = \vspace{4pt}, after = \vspace{4pt}, leftmargin=15pt]
    \item[---] define 
    \review{$\eysbar{\bias}_{k+1}:=\mathbb{E}[\eysbar{E}_k(\bu)]$},
    \item[---] construct 
    \review{$\eysbar{\bV}_{k+1}$} by extracting the $n_{k+1}$ eigenvectors of 
    \review{$\cmatrix[\eysbar{E}_k(\bu)]$} associated to the $n_{k+1}$ largest eigenvalues,
    \item[---] define 
    \review{$\eysbar{E}_{k+1}:\mathbb{R}^{n_k}\to\mathbb{R}^{n_{k+1}}$} by setting 
    \review{$\eysbar{E}_{k+1}(\mathbf{v}):=\act(\eysbar{\bV}_{k+1}^\top(\eysbar{E}_k(\mathbf{v})-\eysbar{\bias}_{k+1}))$}. 
\end{itemize}
Iterate the previous steps for $k=1,\dots,l-1$. Then,
\begin{align}
\label{eq:ineq-chain2}
\inf_{\autoencoder\in\SAEz}\;
\mathbb{E}\|\bu-\mathscr{R}(\autoencoder,\bu)\|^2&\le
\inf_{\autoencoder\in\SBAEz}\;
\mathbb{E}\|\bu-\mathscr{R}(\autoencoder,\bu)\|^2\\\nonumber&\le
\inf_{\autoencoder\in\SOAEz}\;
\mathbb{E}\|\bu-\mathscr{R}(\autoencoder,\bu)\|^2\\\nonumber&\le \sum_{k=0}^{l-1}\lip{\invact}^{2k}\sum_{i=n_{k+1}+1}^{n_k}\lambda_i
\review{(\cmatrix[\review{\eysbar{E}}_k(\bu)])}.
\end{align}
    
\end{corollary}

We highlight that Eqs. \eqref{eq:ineq-chain} and \eqref{eq:ineq-chain2} are complementary, as they provide lower 
and upper bounds for the reconstruction error of the optimal SAE, respectively. Despite being looser than the ones provided in Theorem~\ref{theorem:soae}, such bounds have a practical implication on design and initialization of symmetric autoencoders, as we discuss below.

\subsection{From theory to practice: the EYS initialization}
Inspired by the considerations made in Section~\ref{subsec:theory}, particularly by Corollary~\ref{corollary:loose-bounds}, we propose a new initialization strategy for symmetric autoencoders hereon referred to as ``Eckart-Young-Schmidt (EYS) initialization". The idea is to mimic the greedy procedure described in Corollary~\ref{corollary:loose-bounds}, ultimately translating the latter to an empirical framework.

More precisely, let $\bu\in L^{2}(\mathscr{P},\mathbb{R}^{n_0})$, and assume we are given a collection of $S$ i.i.d. random realizations of $\bu$, denoted as $\bu_1,\dots,\bu_S.$ Fix any bilipschitz activation $\act$ and a suitable skeleton $n_0\ge n_1\ge\dots\ge n_l$ for a symmetric autoencoder. The idea is to leverage the data in order to come up with a suitable autoencoder $\autoencoder\in\hypclass\subseteq\SAE$ that constitutes a first initial guess to be used as a starting point for the training phase. As before, $\hypclass$ denotes the hypothesis class, which can be either the whole set of symmetric autoencoders, or may be limited to either orthogonal or biorthogonal architectures. Our proposed strategy, which is independent of $\hypclass$, is detailed in Algorithm~\ref{alg:eys}.

\begin{algorithm}[h!]
    \caption{EYS initialization (pseudo-code)}
    \label{alg:eys}
    \begin{algorithmic}[0]
    \STATE \textbf{Inputs:} 
    snapshot matrix $\mathbf{U} \in \mathbb{R}^{n_0 \times S}$ listing $S$ i.i.d. samples of $\bu$, bilipschitz activation $\act$, architecture skeleton $\{n_j\}_{j=0}^l$
    \vspace{0.15cm}
    \STATE $\mathbf{Z} \gets \mathbf{U}$ \hfill \texttt{\#initialize latent vector\review{s}}
    \vspace{0.2cm}
    \FOR{$j = 1:l$}
    \STATE $\mathbf{b}_j \gets \overline{\mathbf{Z}}$ \hfill \texttt{\#compute sample mean}\vspace{2pt}
    \STATE $\mathbf{W}_j, \_, \_ = \texttt{SVD}(\mathbf{Z} - \mathbf{b}_j)$ \hfill \texttt{\#obtain left eigenvectors}\vspace{2pt}
    \STATE $\mathbf{V}_j \gets  \mathbf{W}_j[:,:n_j]$ \hfill \texttt{\#retain first $n_j$ vectors}\vspace{2pt}
    \STATE $\mathbf{Z} \gets \act(\mathbf{V}_j^\top(\mathbf{Z} - \mathbf{b}_j))$ \hfill \texttt{\#compute next latent vector\review{s}}
    \ENDFOR
    \vspace{0.2cm}
    
    \STATE $\Psi_{\textsf{EYS}} \gets \{(\bV_j,\bias_j,\act)\}_{j=1}^{l}\in\SOAE$ \hfill \texttt{\#build autoencoder}
    
    \RETURN $\Psi_{\textsf{EYS}}$ 
    \end{algorithmic}
\end{algorithm} 

We highlight that Algorithm~\ref{alg:eys} follows the same construction seen in Corollary~\ref{corollary:loose-bounds}, but it leverages the empirical probability distribution
$\tilde{\mathbb{P}}:=\frac{1}{S}\sum_{i=1}^{S}\delta_{\bu_i},$
rather than the true law $\mathbb{P}$ of the random vector $\bu\sim\mathbb{P}.$ Here, $\delta_{\bu_i}$ denotes the Dirac delta distribution centered at $\bu_i$. Consequently, expected values appearing in Corollary~\ref{corollary:loose-bounds} are replaced by statistical averages, whereas the eigenvectors of the covariance matrices are approximated via empirical SVD (up to centering). In particular, we notice that the EYS initialization ultimately consists in an iterative application of the SVD, progressively going from $n_0$ to $n_l$. \review{Relatedly, it is worth noting that this idea can be understood as a data-driven initialization (or pre-training) strategy, in the same spirit of \cite{hinton2006reducing, al2023pca, seuret2017pca}.}

We also point out that, regardless of the chosen hypothesis class $\hypclass$, the EYS initialization always returns a symmetric orthogonal autoencoder, $\autoencoder_{\textsf{EYS}}\in\SOAE$, coherently with the fact that
$$\SOAE\subseteq\hypclass\subseteq\SAE.$$
However, we also remark that during the training phase the autoencoder will be free to explore the whole hypothesis class, potentially exiting the submanifold of SOAE.

In this concern, it is worth spending a few words in order to clarify how the EYS initialization can be incorporated in the training procedure discussed in Section~\ref{subsec:training}. Let $\autoencoder_{\textsf{EYS}}=\{(\bV_j,\bias_j,\act)\}_{j=1}^{l}$. We distinguish between the three possible choices of the hypothesis class.\vspace{6pt}

\begin{figure}
    \centering
    \includegraphics[width=0.6\linewidth]{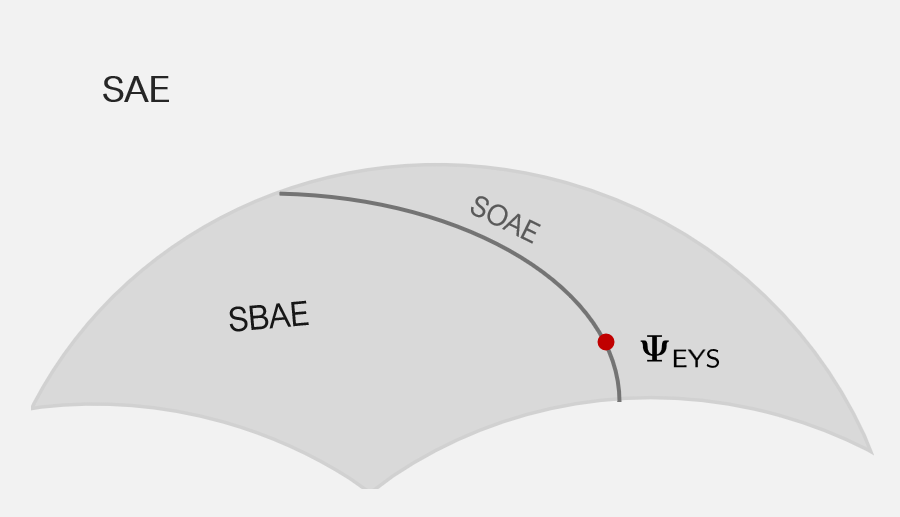}
    \caption{Pictorial representation of the EYS initialization and its relationship with the three manifolds, of decreasing dimension, representing the three (nested) classes of symmetric autoencoder architectures: {\em \textsf{SAE}}, {\em \textsf{SBAE}} and {\em \textsf{SOAE}}.}
    \label{fig:manifold}
\end{figure}

\begin{itemize}
    \item[{\em (i)}] $\hypclass=\SAE$. Here, we just set $\bVE_j:=\bV_j^\top$, $\bVD_j=\bV_j$, $\biasE_j=-\bV_j^\top\bias_j,$ and $\biasD_j=\bias_j$. Then, we optimize $\{(\bVE_j,\bVD_j,\biasE_j,\biasD_j\finaltouches{)\}_{j=1}^{l}}$ freely, as no constraints are involved.\vspace{6pt}

    \item[{\em (ii)}] $\hypclass=\SBAE$. In this case, one must first translate everything according to the parametrization discussed in Section~\ref{subsec:training}, which serves to reframe the training procedure as an unconstrained optimization problem. To this end, we set
    $$\tilde{\mathbf{X}}_j:=[\mathbf{V}_j,\mathbf{M}_j],\quad\tilde{\mathbf{Y}}_j=\tilde{\mathbf{Z}}_j=\mathbf{I}_{n_j},\quad\mathbf{Q}_j=\mathbf{0},\quad\mathbf{s}_j=[\finaltouches{0},\dots,\finaltouches{0}]^\top,$$
    where $\mathbf{I}_{n_j}$ is the $n_j\times n_j$ identity matrix, whereas $\mathbf{M}_j$ is any $n_{j-1}\times d_j$ orthonormal matrix whose span is orthogonal to $\text{span}(\mathbf{V}_j)$. In practice, following the notation in Algorithm~\ref{alg:eys}, this can be constructed by extracting part of the remaning columns in the matrix \review{$\mathbf{W}_j$}, namely, $\mathbf{M}_j=\mathbf{W}_j[:,n_j+1:n_j+d_j]$.

    It is straightforward to prove that with such construction, if we let $$\boldsymbol{\theta}_{\textsf{EYS}}:=\{(\tilde{\mathbf{X}}_j,\tilde{\mathbf{Y}}_j,\tilde{\mathbf{Z}}_j,\mathbf{Q}_j,\mathbf{s}_j,\bias_j)\}_{j=1}^l$$
    then $\autoencoder_{\boldsymbol{\theta}_{\textsf{EYS}}}=\autoencoder_{\textsf{EYS}}$, meaning that the proposed parameter vector is consistent with the proposed initialization. Consequently, the training phase can now be carried out in the optimization space starting from $\boldsymbol{\theta}=\boldsymbol{\theta}_{\textsf{EYS}}$.
    \vspace{6pt}

    \item[{\em (iii)}] $\hypclass=\SOAE$. Here, as in the previous case, one needs to find a suitable $\boldsymbol{\theta}_{\textsf{EYS}}\in\Theta$ in the optimization space that realizes $\autoencoder_{\textsf{EYS}}$ ---this time according to the notation for SOAE architectures: see Section~\ref{subsec:training}. Here, however, this is much easier to do as one can simply set $\boldsymbol{\theta}_{\textsf{EYS}}:=\{(\bV_j,\bias_j)\}_{j=1}^{l}$.\vspace{6pt}
\end{itemize}

In other words, our proposal is to use the same initialization for all classes of symmetric autoencoders: the choice of the hypothesis class will only impact the architecture during the training phase, as it will restrict the exploration to the corresponding submanifold. We refer to Figure~\ref{fig:manifold} for a pictorial representation. 
Finally, we remark that, even if the EYS initialization has been developed specifically for symmetric autoencoders, in principle, it can also be applied to generic autoencoder architectures. However, the underlying reasoning would then lose its theoretical significance.

\section{Numerical experiments}
\label{sec:numerical-experiments}
Within this section, we provide a comprehensive series of numerical experiments aimed at {\em (i)} assessing the accuracy and computational performance of deep symmetric autoencoders, in their three flavours (\textsf{SAE}, \textsf{SBAE}, \textsf{SOAE}), {\em (ii)} validating the efficacy of our proposed initialization strategy, 
and {\em (iii)} unveiling the effect of design choices on the reconstruction accuracy. The training of the architectures is performed on a NVIDIA \review{L40S 40GB} 
GPU. The code implementation is released and available at \texttt{\href{https://github.com/briviosimone/sae_eys}{https://github.com/briviosimone/sae\_eys}}.

\subsection{Datasets description} 

For our analysis, we shall consider data coming from three different case studies. 
All \review{problems} 
are formulated within the framework of \emph{discretized functional data}, \review{which we consider 
\finaltouches{propaedeutic to} applications involving reduced order modeling of parametrized PDEs}. Specifically, the input dimension $n_0$ arises from the discretization of a spatial domain $\Omega \subset \mathbb{R}^d$, obtained by considering either a uniform grid or an appropriate simplicial mesh. Within this setting, the state vector $\bu\in\mathbb{R}^{\ndim}$ ultimately serves as a discrete representation of some underlying functional object. That is, being $\{x_i\}_{i=1}^{\ndim}$ the nodes used for the spatial discretization, one has $u_i=u(x_i)$ for some $u:\Omega\to\mathbb{R}$. Specifically, in our case, $\bu$ will be the discrete vector representing the solution of some parameter dependent PDE, and the variability of the solution $\bu$ will be obtained by letting the parameters of the PDE vary within a suitable range: see Table~\ref{tab:params} and Fig.~\ref{fig:dataset-variability}.

Consequently, we emphasize that all the data is synthetic and is generated using either analytic formulae or numerical solvers. We exploit this controlled environment to ensure a fair comparison of the different architectures across the different case studies. 
For each numerical experiment, we collect \finaltouches{800} realizations of the state vector $\bu$, generated by repeatedly sampling different values of the problem parameters. Each parameter is sampled independently and uniformly over its respective physical range (see Table~\ref{tab:params}).
After computing the $S_{\text{tot}}=\review{800}$ simulations, we split them between training (\review{$1/4$} of the total), validation (\review{$1/8)$} and test set (\review{$5/8)$}. To guarantee a fair comparison, we conduct all the experiments following a standardized training routine. Specifically, we employ 1500 epochs (with a patience of 500 for early stopping), using \review{Adam \cite{kingma2017adammethodstochasticoptimization} with} $lr = 10^{-3}$ as learning rate, $B=8$ as batch size, and min-max normalization for all of our experiments. The simulations in the test set, instead, are used to compute the empirical mean-squared error, $\text{MSE}_{\finaltouches{\text{test}}}=\frac{1}{S_{\text{test}}}\sum_{j=1}^{S_{\text{test}}}\|\bu_j^{\text{test}}-\mathscr{R}(\autoencoder,\bu_j^{\text{test}})\|^2,$
which serves as a Monte Carlo approximation of $\mathbb{E}\|\bu-\mathscr{R}(\autoencoder,\bu)\|^2.$

We provide a brief description of the three case studies right below.
\vspace{0.2cm}

\begin{figure}
    \centering
    \includegraphics[width=\linewidth]{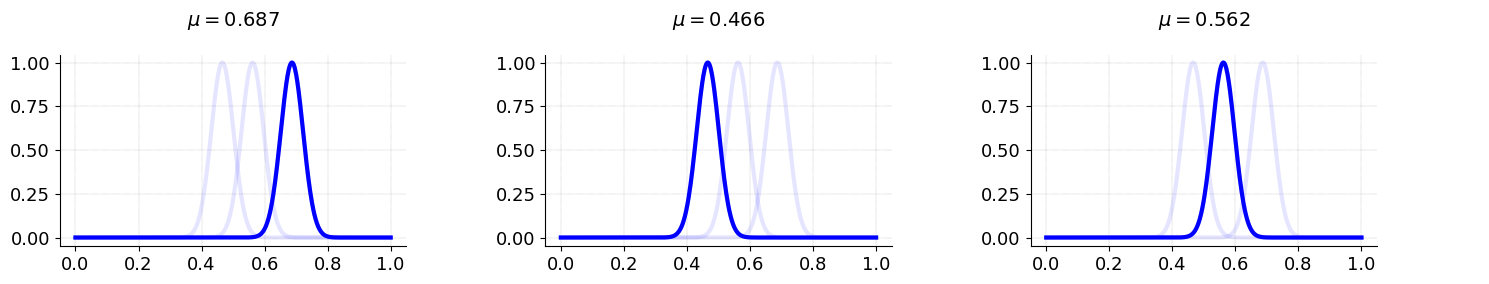}
   \\ \vspace{9mm}
    \includegraphics[width=\linewidth]{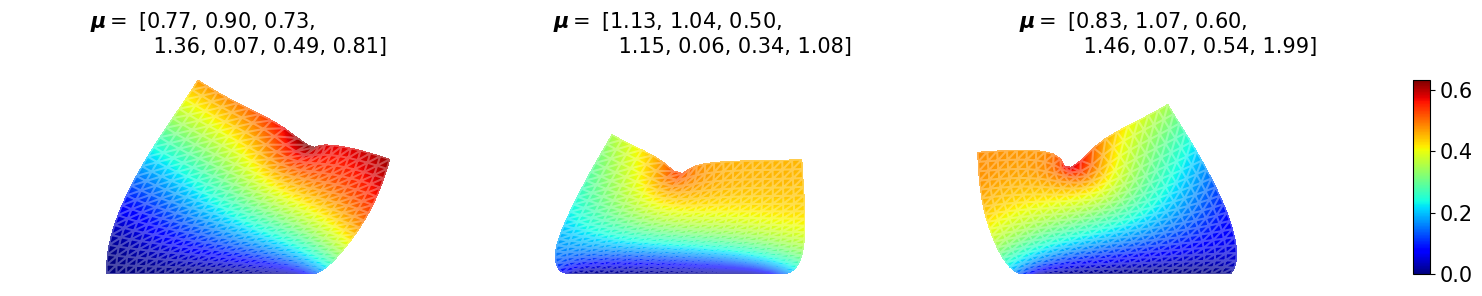} \\ \vspace{9mm}
     \includegraphics[width=\linewidth]{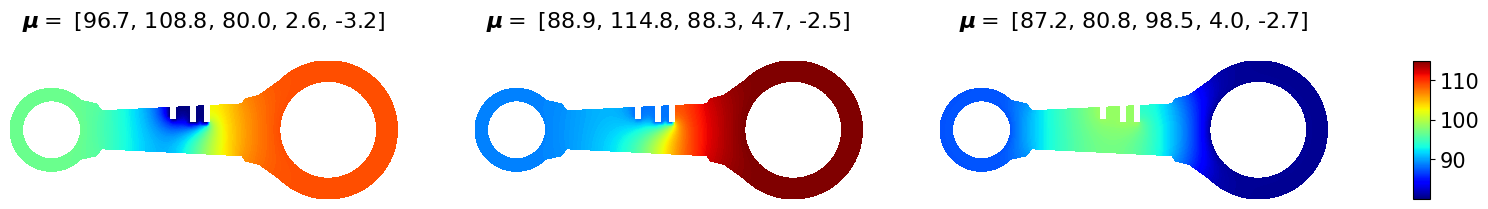}
    \caption{Random samples of $\bu$ for the three test cases considered in this work: \textnormal{\textsf{PGA}} (top), \textnormal{\textsf{ELS}} (middle), and \textnormal{\textsf{ROD}} (bottom). We denote with $\bmu$ the collection of problem parameters associated to each simulation.}
    \label{fig:dataset-variability}
\end{figure}

\renewcommand{\arraystretch}{1.3}
\begin{table}
\centering
\footnotesize{\;\\\;\\\;\\\;\\
\begin{tabular}{llll}
\toprule
\bf {Case study} & \bf {Parameter}  & \bf {Range} & \bf {Description} \\
\midrule
\textsf{PGA} & $\mu$ & [0.3, 0.7] & wave position  \\
 \hline
 \multirow{7}{*}{\textsf{ELS}} & $\rho$ & [0.5, 2] & body force
 \\
 & $\lambda$ & [0.9, 1.1] & first Lamé coefficient \\

 & $\mu$ & [0.5, 1] & second Lamé coefficient \\
 
 & $m$ & [\review{1.0}, \review{1.5}] & mass of the falling object \\

 & $\delta$ & [0.05, 0.1] & size of the falling object\\

 & $x_0$ & [0.3, 0.7] & impact location\\

 & $\theta$ & [$\frac{\pi}{4}$, $\frac{3\pi}{4}$] & collision angle\\
 \hline
 \multirow{5}{*}{\textsf{ROD}}  & $T_1$ & [80, 120] & temperature at the inner left circle \\
 
 & $T_2$ & [80, 120] & temperature at the inner right circle \\

 & $T_3$ & [80, 120] & temperature at the toothed region \\

 & $\log_{10}f$ & [1.5, 5] & source term (log scale)\\

 & $\log_{10}\alpha$ & [-3.5, -1.5] & irradiation coefficient (log scale)\\
 
\bottomrule
\end{tabular}
}
\vspace{6pt}
\caption{\label{tab:params}Problem parameters for the three case studies considered in the numerical experiments. The variability of such parameters is responsible of the stochasticity of the high-dimensional datum $\bu$.}
\end{table}

\paragraph{Parameterized \review{Gaussian} \textnormal{(\textsf{PGA})}} As the first test case, we select a classical experiment setup stemming from the reduced order modeling community, namely, the one-dimensional parameterized \review{Gaussian} function,
\begin{equation*}
    u(x;\mu) = \exp(-400(x-\mu)^2), \quad x \in [0,1],
\end{equation*}
having $\mu \in [0.3,0.7]$ as a unique parameter. The latter can be thought as the solution at terminal time $T=1$ of the advection equation $\partial_{t}u+\mu\partial_x u=0$ with a fixed initial profile. For the sake of the discretization we use a uniform grid of stepsize $h=1/513$, resulting in a problem dimension of $n_0=514$. \vspace{0.2cm}

\paragraph{Linear elasticity \textnormal{(\textsf{ELS})}} For this case study, we consider a 
PDE model describing the collision between a falling object and a block of elastic material, 
\begin{equation}
\label{eq:elasticity}
\left\{
\begin{array}{rll}
- \nabla \cdot \left (\lambda\nabla\cdot\boldsymbol{u} + \mu\left(\nabla\boldsymbol{u}+\nabla\boldsymbol{u}^{\top}\right)\right) &= \boldsymbol{f} \qquad & \text{in }\Omega\\
\boldsymbol{u}&=\mathbf{0} \qquad& \text{on }\Gamma_{\text{bottom}}\\
\finaltouches{\lambda(\nabla\cdot \boldsymbol{u})\cdot\boldsymbol{n}+\mu\left(\nabla\boldsymbol{u}+\nabla\boldsymbol{u}^{\top}\right)\cdot\boldsymbol{n}}&=\finaltouches{\boldsymbol{T}} \qquad & \text{on }\Gamma_{\text{top}}\\
\lambda(\nabla\cdot \boldsymbol{u})\cdot\boldsymbol{n}+\mu\left(\nabla\boldsymbol{u}+\nabla\boldsymbol{u}^{\top}\right)\cdot\boldsymbol{n}&=\mathbf{0} \qquad & \text{on }\partial\Omega\setminus (\Gamma_{\text{bottom}}\cup\Gamma_{\text{top}}).
\end{array}
\right.
\end{equation}
Here, $\Omega = (0,1)^2$ is the spatial domain, $\Gamma_{\text{bottom}}$ and $\Gamma_{\text{top}}$ are the bottom and top edge, while $\boldsymbol{T}(x,y)=\boldsymbol{1}_{[x_0-\delta,\;x_0+\delta]}(x)\cdot[\cos\theta,-\sin\theta]^\top m$ and $\boldsymbol{f}=[0, -\rho]^\top$ are the imposed \review{traction} at the top edge and the forcing term, respectively. The problem depends on 7 parameters, modeling the mechanical properties of the absorbing block and the falling object: we refer to Table~\ref{tab:params} for a short description about their meaning.

We approximate the solution to \eqref{eq:elasticity} using continuous piecewise-linear finite elements  ($\mathbb{P}^1$-FEM) over a a mesh consisting of 441 vertices. Notice that, since the solution to \eqref{eq:elasticity} is vector-valued, the actual problem dimension is $n_0=2\times441=882.$ \finaltouches{We remark that, under the parametric ranges considered in Table \ref{tab:params}, the displacement fields obtained by solving \eqref{eq:elasticity} end up violating the small deformation hypothesis, thus resulting in solutions with poor physical significance.
Still, because our problem formulation is effectively adimensional, physical significance can be recovered by suitably rescaling the parameters that contribute to the right hand side of the variational problem. Here, we deliberately refrain from doing so in order to make the resulting deformations more clearly visible.}

\vspace{0.2cm}

\paragraph{Nonlinear diffusion in a rod geometry \textnormal{(\textsf{ROD})}} To generate the third dataset, we solve repeatedly a nonlinear PDE describing the temperature distribution in a non-trivial rod geometry, for different configurations of the governing parameters: see Table~\ref{tab:params}. The PDE reads as follows,
\begin{equation}
\label{eq:rod}
\left\{
\begin{array}{rll}
    -\nabla\cdot(\sigma(u)\nabla u) &= f-\alpha u^{4} \qquad& \text{in}\;\Omega\\
    \sigma(u)\nabla u \cdot \boldsymbol{n} &= 0 \qquad & \text{on}\;\partial\Omega\setminus\left(\cup_{i=1}^{3}\Gamma_{i}\right)\\
    u &= T_{i} \qquad & \text{on}\;\Gamma_{i},
\end{array}
    \right.
\end{equation}
where $\sigma(u):=10^3 + e^{u/8}$ is the diffusion coefficient, whereas $\Gamma_{1},\Gamma_{2},\Gamma_{3}$ are the inner left circle, inner right circle, and the ``toothed region", respectively: see Fig.~\ref{fig:dataset-variability}. We approximate the solution to \eqref{eq:rod} using $\mathbb{P}^1$-FEM with $n_0=4347$ degrees of freedom, distributed across a nonstructured triangular mesh.

\subsection{Initialization}

In the following, we validate the efficacy of our initialization strategy by comparing it with the standard ones proposed in the literature. 
In the case of unconstrained architectures, we consider the so-called ``He initialization" \cite{he2015delvingdeeprectifierssurpassing} as our main competitor. Specifically,
we use the original algorithm 
when $\act$ is the $(\alpha,\beta)-$LeakyReLU. 
\finaltouches{Conversely, we rely a suitable generalization of the latter when considering other activations. We mention, in fact, that}
it is possible to extend the underlying reasoning put forth in \cite{he2015delvingdeeprectifierssurpassing} to other bilipschitz activations: we defer \finaltouches{such discussion} to the Appendix~\ref{appendix:init}.

In the case of autoencoders with (bi)-orthogonality constraints, instead, we consider  the initialization proposed by Otto et al. \cite{otto2023learning} as a benchmark. The latter consists in randomly sampling 
the weight matrices from the orthogonal group, while initializing all bias vectors to zero.

In order to assess the effectiveness of the EYS initialization, we investigate its impact on: \emph{(i)} the \finaltouches{validation} MSE before training, to see whether it provides a good initial guess; \emph{(ii)} the training dynamics, to see whether it affects the final accuracy and the speed of convergence. We summarize our findings below.
\vspace{6pt}
\paragraph{Before training} For the sake of brevity, we limited this preliminary analysis to the \textsf{PGA} case study: the \textsf{ELS} and \textsf{ROD} test cases will be addressed in the next paragraph. We report our findings in 
Figure~\ref{fig:init_study}, 
where we explore how different \textsf{SAE} architectures of variable latent dimension $n_l$ and depth $l$ perform before training, comparing their MSE under the EYS and the He initialization, respectively.
The analysis, here reported relatively to the HypAct$_\theta$ activation, was repeated for different values of the $\theta$-parameter, corresponding to different sharpness values of the nonlinearity $\act=\act_\theta$.

We observe that, as expected, the latent dimension does not influence the initial MSE value in the ``standard" initialization, regardless of the sharpness of the activation function. Conversely, being data-driven, the EYS initialization shows a strong dependency between the starting MSE and the latent dimension. In particular, for larger values of $n_l$ the EYS initialization performs significantly better ---by several orders of magnitude--- when compared to the literature standard. Notably, this improvement is found consistently for all sharpness values. 

The MSE of a \textsf{SAE} initialized with the EYS initialization, however, tends to deteriorate as the architecture gets deeper and deeper, especially when combined with activation functions entailing a larger sharpness. Nonetheless, this is consistent with our estimates in Theorem~\ref{theorem:soae}, which ultimately suggest an exponential dependency of the error with respect to $l$, with a growth rate depending on $\angle(\act_\theta) = \angle(\theta).$ In contrast, the He initialization shows no dependency on the depth of the architecture, coherently with the rationale underlying its construction \cite{he2015delvingdeeprectifierssurpassing}.

We mention that similar results were obtained for the $(\alpha,\beta)$-LeakyReLU, as well as for the other autoencoder architectures (\textsf{SBAE}, \textsf{SOAE}). We \finaltouches{direct} the interested reader to the Supplementary Materials.

\vspace{6pt}

\begin{figure}
    \centering
    \includegraphics[width=0.99\linewidth]{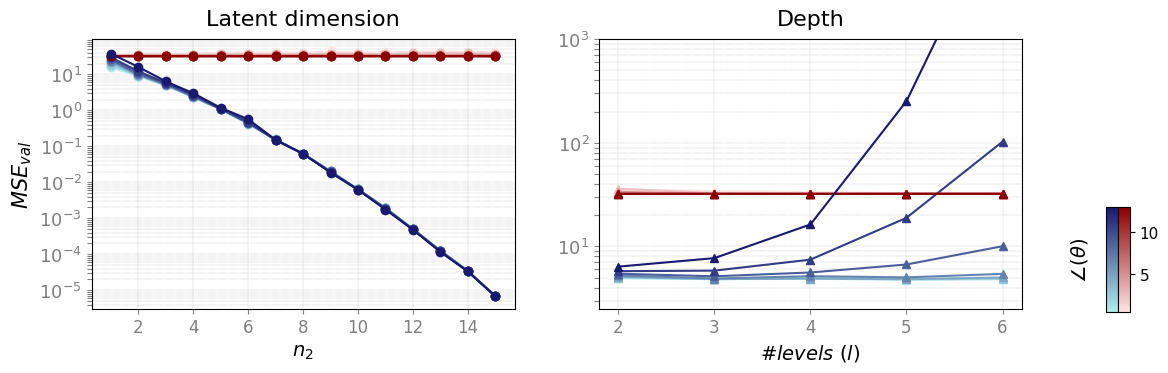}\vspace{-10pt}
    \caption{Comparison between the EYS {\em ({\color{darkblue}\textbf{blue}})} and the He initialization {\em ({\color{darkred}\textbf{red}})} for the {\em\textsf{PGA}} case study. Models are compared in terms of the \textbf{initial} MSE over the \finaltouches{validation} set. Results refer to {\em \textsf{SAE}}s with $HypAct_{\theta}$ activation of varying sharpness $\angle=\angle(\theta)$. Left: two-level architecture $\{n_0,n_1,n_2\}$ with $n_1=20$ and varying latent dimension $n_2$. Right: {\em \textsf{SAE}}s of variable depth but fixed latent dimension, obtained following the pattern $\{n_0, 65,3\}, \{n_0,65,5,3\}, \ldots, \{n_0, 65,33,17,9,5,3\}$. NB: to account for the stochasticity of the He initialization, we repeatedly initialized the architectures 100 times and only kept the trial reporting the best performance.}
    \label{fig:init_study}
\end{figure}

\paragraph{During and after training}
To further validate the efficacy of our strategy, we ought to demonstrate the impact of our proposed initialization on the training dynamics, this time for a fixed architecture design (cf. Table~\ref{tab:architecture-skeleton}). 
\review{In order to gain additional insights, we further enrich our analysis by including a classic autoencoder (\textsf{AE}) as a model baseline. We construct the latter using same skeleton adopted for the other architectures,  but fixing the activation function to be the same for the encoder and the decoder. More precisely, we set $\rho_{AE}^{dec} = \rho^{enc}_{AE} = \rho_{SAE}^{-1}$, so that the decoders of \textsf{AE} and its symmetric counterpart share the exact same structure.}

Our findings are summarized in Fig.~\ref{fig:comparison}. 
There, we see that the EYS procedure 
\review{generally} outperforms its standard counterpart, ensuring both a faster convergence and a  smaller loss at the end of the training phase. 
\review{This is particularly true in the ELS and ROD test cases, where all}
autoencoder architectures, \textsf{SAE}, \textsf{SBAE} and \textsf{SOAE}, \review{show a substantial improvement. Things, however, are not as evident in the PGA example, where, even though the EYS-initialized architectures start the training with lower loss values when compared to their standard-initialized counterparts, it seems that this advantage does not make a difference in term of training convergence. We hypothesize that is caused by the intrinsic nature of the PGA test case, which is known to be particularly troublesome for linear reduction methods, upon which the EYS initialization is based}. We mention that, while Fig.~\ref{fig:comparison} only shows the results for the $(\alpha,\beta)$-LeakyReLU with \finaltouches{a} specific choice \review{of $\alpha$ and $\beta=5/4$}, analogous results were found \review{for a different value of $\alpha$, and for the HypAct$_\theta$ activation}. We refrain from reporting here such results in order to minimize redundancy, and we refer the interested reader to the Supplementary Materials.

\begin{figure}[htb!]
    \centering
    \includegraphics[width=\linewidth]{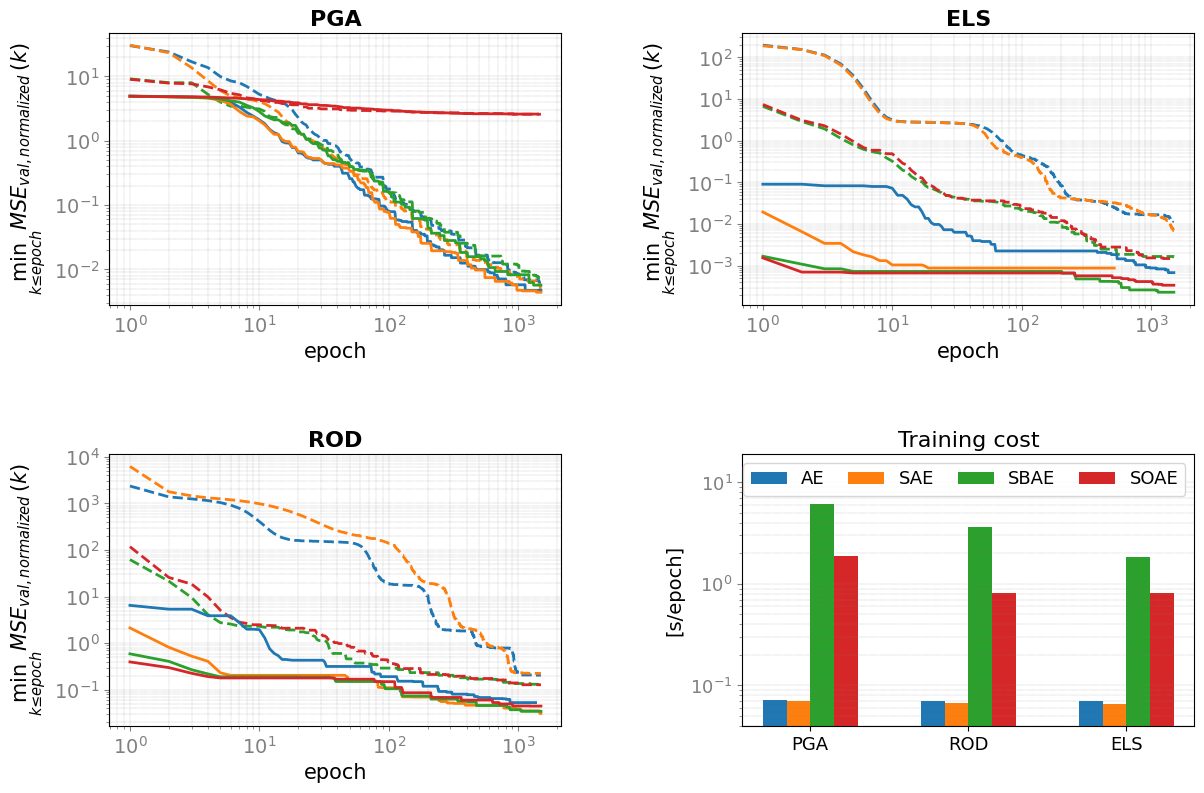}\vspace{-10pt}
    \caption{Visualization of the training dynamics and computational performance of {\em \textsf{AE}, \textsf{SAE}, \textsf{SBAE}, and \textsf{SOAE}}, for all the test cases. We compare our proposed initialization {\em (solid line)} with the literature standard {\em (dashed line)}. The reported result refer to the $\textnormal{LeakyReLU}_{\alpha,5/4}$ activation, with $\angle(\alpha) = 0.5$. \review{The validation MSE is computed on normalized data}.
    We refer the reader to Table~\ref{tab:architecture-skeleton} for the specifics on the employed architectures.}
    \label{fig:comparison}
\end{figure} 

\renewcommand{\arraystretch}{1.2}
\begin{table}[htb!]
\centering
\small{
\begin{tabular}{lccc}
\toprule
\bf {Test case} & \sf {PGA}  & \sf {ELS} & \sf {ROD} \\
\midrule
\bf {Skeleton} $\{n_j\}_{j=0}^l$ & \{514,\; 64,\;15,\;3\} & \{882,\;15,\;10,\;8\} & \{4347, \;15,\;7,\;5\} \\
\bottomrule
\end{tabular}
}
\vspace{6pt}
\caption{\label{tab:architecture-skeleton}Architecture skeletons for the comparative study. We use the criteria proposed in \cite{franco2023deep} to set the latent dimension, while we use the error bounds developed in this paper to heuristically set the other levels' complexity. }
\end{table}

\subsection{Comparative study} Hereinafter, we compare \textsf{SAE}, \textsf{SBAE}, \textsf{SOAE} in terms of training efficiency and reconstruction accuracy, varying critical parameters such as the activation function and its sharpness value. 
Our findings are summarized in Table~\ref{tab:rod-table} and Fig.~\ref{fig:comparison}. 

As a first notice, we observe that, coherently with our previous analysis, the EYS initialization proves superior for \finaltouches{nearly} all autoencoder architectures.
Surprisingly, this is also true for the classical \textsf{AE}.
This could be explained by noting that, after initialization, the decoders of \textsf{AE} and \textsf{SAE} are effectively identical. Still, after training, \textsf{SAE} typically achieves better performances, supporting the use of different activation functions for encoding and decoding: see Table~\ref{tab:rod-table}. Furthermore, during training, \textsf{SAE} converges much faster than \textsf{AE}, while also entailing a smaller computational cost, Fig.~\ref{fig:comparison}. \finaltouches{Such cost reduction can} be explained by the fact that, at the bottleneck, \textsf{AE} has an additional activation compared to \textsf{SAE}. 

What is most interesting, however, is that \textsf{SBAE} often outperforms \textsf{SAE}, meaning that the biorthogonality constraints not only do not hinder the expressivity of the architecture, but they can even prove beneficial. Since $\SBAE\subseteq\SAE$, this \finaltouches{may indicate} 
that the loss landscape is better behaved on the submanifold of \textsf{SBAE} than on the larger space of \textsf{SAE}, resulting in a more regular optimization.
On the other hand, these benefits come at a larger computational cost: \textsf{SBAE}, in fact, is the most expensive architecture to be trained, as seen in Fig.~\ref{fig:comparison}.

An interesting compromise is given by \textsf{SOAE}, which, although still more onerous than a classical unconstrained autoencoder, is effectively cheaper than a \textsf{SBAE} whilst still offering competitive performances. However, we also notice that the tight architecture restrictions characterizing \textsf{SOAE}s may occasionally prevent \review{these} architectures from generating accurate reconstructions: see \textsf{PGA} test case, Fig.~\ref{fig:comparison}.
\review{Before coming to the conclusions, it is worth making a final note about the robustness of the results presented so far. Specifically, it is important to assess whether these findings remain consistent when models are trained under different conditions—such as varying sample sizes (e.g., in low-data regimes), alternative architectural configurations (e.g., different numbers of neurons or layers), or different optimization hyperparameters (e.g., learning rates).
To this end, the interested reader can find an extensive series of numerical experiments, corroborating the previous analysis across these diverse settings, in the Supplementary Materials, Sections SM3 and SM4.}

\renewcommand{\arraystretch}{1.2}
\begin{table}[htb!]
\centering
\begin{footnotesize}
\renewcommand{\arraystretch}{1.25}
\begin{tabular}{lllll}
\toprule
 &  &  & \textbf{MSE} & \textbf{MRE} \\
\midrule
\multirow[c]{8}{*}{$\text{HypAct}_{\theta}$} & \multirow[c]{4}{*}{$\angle(\theta) = 0.5$} & $\textsf{AE}$ & $(1.09 \pm 0.16) \times 10^{2}$ & $(1.36 \pm 0.07) \times 10^{-3}$ \\
 &  & $\textsf{SAE}$ & $(1.25 \pm 0.15) \times 10^{2}$ & $(1.52 \pm 0.07) \times 10^{-3}$ \\
 &  & $\textsf{SBAE}$ & $(5.28 \pm 0.76) \times 10^{1}$ & $(9.54 \pm 0.50) \times 10^{-4}$ \\
 &  & $\textsf{SOAE}$ & $(1.54 \pm 0.22) \times 10^{2}$ & $(1.62 \pm 0.09) \times 10^{-3}$ \\
\cline{2-5}
 & \multirow[c]{4}{*}{$\angle(\theta) = 3.0$} & $\textsf{AE}$ & $(2.65 \pm 0.52) \times 10^{2}$ & $(2.09 \pm 0.13) \times 10^{-3}$ \\
 &  & $\textsf{SAE}$ & $(1.28 \pm 0.14) \times 10^{2}$ & $(1.56 \pm 0.06) \times 10^{-3}$ \\
 &  & $\textsf{SBAE}$ & $(3.55 \pm 0.72) \times 10^{1}$ & $(7.34 \pm 0.47) \times 10^{-4}$ \\
 &  & $\textsf{SOAE}$ & $(1.27 \pm 0.20) \times 10^{2}$ & $(1.46 \pm 0.08) \times 10^{-3}$ \\
\cline{1-5} 
\multirow[c]{8}{*}{$\text{LeakyReLU}_{\alpha,5/4}$} & \multirow[c]{4}{*}{$\angle(\alpha) = 0.5$} & $\textsf{AE}$ & $(1.14 \pm 0.14) \times 10^{2}$ & $(1.44 \pm 0.07) \times 10^{-3}$ \\
 &  & $\textsf{SAE}$ & $(6.22 \pm 1.15) \times 10^{1}$ & $(9.98 \pm 0.59) \times 10^{-4}$ \\
 &  & $\textsf{SBAE}$ & $(7.73 \pm 1.32) \times 10^{1}$ & $(1.13 \pm 0.06) \times 10^{-3}$ \\
 &  & $\textsf{SOAE}$ & $(1.03 \pm 0.18) \times 10^{2}$ & $(1.26 \pm 0.08) \times 10^{-3}$ \\
\cline{2-5}
 & \multirow[c]{4}{*}{$\angle(\alpha) = 3.0$} & $\textsf{AE}$ & $(9.01 \pm 1.91) \times 10^{2}$ & $(3.69 \pm 0.27) \times 10^{-3}$ \\
 &  & $\textsf{SAE}$ & $(1.25 \pm 0.13) \times 10^{2}$ & $(1.54 \pm 0.06) \times 10^{-3}$ \\
 &  & $\textsf{SBAE}$ & $(5.19 \pm 0.73) \times 10^{1}$ & $(9.46 \pm 0.50) \times 10^{-4}$ \\
 &  & $\textsf{SOAE}$ & $(1.68 \pm 0.22) \times 10^{2}$ & $(1.73 \pm 0.09) \times 10^{-3}$ \\
\bottomrule
\end{tabular}

\renewcommand{\arraystretch}{1}
\end{footnotesize}
\vspace{6pt}
\caption{\label{tab:rod-table}\review{Average test} reconstruction accuracy \review{(with 95\% confidence interval)} of {\em \textsf{AE}, \textsf{SAE}, \textsf{SBAE}, and \textsf{SOAE}} for the {\em \textsf{ROD}} test case, \review{when considering the EYS initialization, but} varying the activation function and the sharpness parameter. }
\end{table}

\section*{Conclusions}
\label{sec:conclusions}
Within the present work, we devoted our efforts to the study of deep symmetric autoencoders. Specifically, we put forth
\vspace{0.15cm}
\begin{itemize}[itemsep=4pt]
 \item[{\footnotesize$\bullet$}] \emph{a comprehensive mathematical framework} for their classification, analysis, and implementation. In particular, we first distinguished between constrained and unconstrained architectures. Under this scope, we characterized their numerical properties, focusing on model stability and representation consistency. To complete the picture, we devised new practical strategies for their parameterization and optimization;
 \item[{\footnotesize$\bullet$}] \emph{a novel initialization strategy}, termed EYS, directly stemming from newfound error bounds that connect deep symmetric autoencoders back to linear reduction methods. To the best of our knowledge, this is the first \textit{data-driven} strategy that leverages low-rank structures for the initialization of \review{nonlinear symmetric} autoencoders \review{with bilipschitz activations}.
\end{itemize}
\vspace{0.15cm}

In our experiments, the EYS initialization reported promising results, suggesting a potentially broader impact than the one explored in this work. 
Indeed, we believe that plenty of exciting research directions may originate from here, both from a theoretical and a practical standpoint.
For instance, 
future research may focus on deepen our theoretical understanding of the relation between the EYS optimization and the optimization dynamics, possibly investigating a connection with \cite{nguyen2020benefitsjointlytrainingautoencoders}. 
Another interesting question is whether one could craft \review{a} suitable extension of the proposed initialization to be used for more complex and competitive architectures, such as convolutional autoencoders \cite{pant2021deep} or graph-based architectures \cite{pichi2024graph}. 
Finally, one could exploit the proposed framework for the construction of neural surrogates for the simulation of complex physical systems, viz. DL-ROMs \cite{fresca2021comprehensive}.

On top of this, future research could also be devoted at addressing some of the limitations \review{of} this work, such as the exploration of applications beyond model order reduction, \review{like}, e.g., image compression \cite{balle2017endtoend}, \review{the integration of probabilistic methods, such as randomized SVD, to handle very high-dimensional problems \cite{halko2011finding} and improve the scaling of the EYS initialization,} or the derivation of suitable theoretical results \review{capable of bridging} model expressivity and training complexity, as in \cite{adcock2021gap, de2023convergence}.

\subsection*{Acknowledgments}
SB acknowledges the support of European Union - NextGenerationEU within the Italian PNRR program (M4C2, Investment 3.3) for the PhD Scholarship ``Physics-informed data augmentation for machine learning applications".
\review{SB and NF acknowledge} the support of the project \textit{Reduced Order Modeling and Deep Learning for the real-time approximation of PDEs (DREAM)}, grant no. FIS00003154, funded by the Italian Science Fund (FIS) and by Ministero dell'Università e della Ricerca (MUR).
Part of the simulations discussed in this work were performed on the HPC Cluster of the Department of Mathematics of Politecnico di Milano which was funded by MUR grant ``Dipartimento di Eccellenza" 2023-2027. The authors are members of the Gruppo Nazionale Calcolo Scientifico-Istituto Nazionale di Alta Matematica (GNCS-INdAM).

\appendix
\section{Technical proofs} 
\label{appendix:proofs}

\begin{proof}[\nopunct Proof of Theorem~\ref{theorem:pod}]
Recall that for a generic random vector $\mathbf{z}\in\mathbb{R}^{m}$ one has
$$\mathbb{E}\|\mathbf{z}-\mathbf{d}\|^2\ge \mathbb{E}\|\mathbf{z}-\mathbb{E}[\mathbf{z}]\|^2=\text{Tr}(\cmatrix[\mathbf{z}]),$$
for all deterministic vectors $\mathbf{d}\in\mathbb{R}^m$. Then, it is straightforward to prove that $\bq_*=\mathbb{E}[\bu]$ provides an optimal shifting. Indeed, for all  $\bV\in\mathcal{O}_{\ndim, n}$ and all $\bq\in\mathbb{R}^{n_0}$ one has
$$\mathbb{E}\|\bu-\left[\bV\bV^\top(\bu-\bq)+\bq\right]\|^2=\mathbb{E}\|(\mathbf{I}-\bV\bV^\top)\bu - (\mathbf{I}-\bV\bV^\top)\bq\|^2.$$
Since $\mathbb{E}\left[(\mathbf{I}-\bV\bV^\top)\bu\right]=(\mathbf{I}-\bV\bV^\top)\mathbb{E}\left[\bu\right]$, up to setting $\mathbf{z}:=(\mathbf{I}-\bV\bV^\top)\bu$, it is evident that $\bq_*=\mathbb{E}[\bu]$ is optimal, regardless of the matrix $\bV.$
Next, we notice that $\tilde{\bu}:=\bu-\mathbb{E}[\bu]$ and $\bu$ share the same covariance matrix, $\cmatrix[\tilde{\bu}]=\cmatrix[\bu].$ Therefore, without loss of generality, we shall continue the proof assuming that $\mathbb{E}[\bu]\equiv0$. 

Let now $T_\bu=\sum_{i=1}^{r}s_i\bv_i\langle\xi_i,\cdot\rangle_{L^{2}(\mathscr{P})}$ be the SVD decomposition of $T_\bu$, given as Theorem~\ref{theorem:EYS}.
Notice that, up to considering $\cmatrix[\bu]$ as an operator from $\mathbb{R}^{\ndim}\to\mathbb{R}^{\ndim}$, one has
$\cmatrix[\bu]=T_\bu T_\bu^*.$ In fact, for all $\mathbf{a},\mathbf{b}\in\mathbb{R}^{\ndim}$,
\begin{multline*}
\mathbf{a}^\top T_\bu T_\bu^*\mathbf{b}=\mathbf{a}^\top \mathbb{E}[\bu T_\bu^*\mathbf{b}]=\mathbb{E}[\mathbf{a}^\top\bu T_\bu^*\mathbf{b}]=\langle\mathbf{a}^\top\bu,T_\bu^*\mathbf{b}\rangle_{L^2(\mathscr{P})}=\\=\langle T_\bu\mathbf{a}^\top\bu,\mathbf{b}\rangle_{\mathbb{R}^{\ndim}}=\mathbb{E}[\bu\mathbf{a}^\top\bu]^\top\mathbf{b}=\mathbb{E}[(\bu^\top\mathbf{a})(\bu^\top\mathbf{b})]=
\mathbf{a}^\top\cmatrix[\bu]\mathbf{b}.\end{multline*}
It follows that, up to multiplicities, $s_i^2=\lambda_i(\cmatrix[\bu])$ and $\mathbf{X}=[\bv_1^*,\dots,\bv_r^*]$, as the eigenvalues and eigenvectors of $\cmatrix[\bu]$ are nothing but the squared singular values and left singular vectors of $T_\bu$.

As a next step, given any $\bV\in\mathscr{O}_{\ndim,n}$, let  $P_{\bV}:\mathbb{R}^{\ndim}\to\mathbb{R}^{\ndim}$ be the projector acting as $P_{\bV}:\boldsymbol{y}\mapsto\bV\bV^\top\boldsymbol{y}.$ Notice that, by linearity,
$T_{P_{\bV}\bu}=P_{\bV}T_\bu.$
In particular, $\rank(T_{P_{\bV}})\le n$, and, by Proposition~\ref{prop:isometry} and Theorem~\ref{theorem:EYS},
$$\mathbb{E}\|\bu-\bV\bV^\top\bu\|^2=\|T_\bu-T_{P_\bV\bu}\|_{\textnormal{HS}}^2\ge\sum_{i=n+1}^r\lambda_i(\cmatrix[\finaltouches{\bu}]).$$
At the same time, if $\bV=\bV_*$, we get
$
T_{P_{\bV_*\bu}}=P_{\bV_*}T_\bu = P_{\bV_*}\sum_{i=1}^{r}s_i\bv_i\langle\xi_i,\cdot\rangle_{L^2(\mathscr{P})}=\\=\sum_{i=1}^{r}s_iP_{\bV_*}\bv_i\langle\xi_i,\cdot\rangle_{L^2(\mathscr{P})}=\sum_{i=1}^{n}s_i\bv_i\langle\xi_i,\cdot\rangle_{L^2(\mathscr{P})},$
i.e.: $T_{P_{\bV_*\bu}}$ is 
the $n$-rank truncated SVD of $T_\bu$. 
Thus,
$
    \mathbb{E}\|\bu-\bV_*\bV_*^\top\bu\|^2 = \|T_\bu-T_{P_{\bV_*}\bu}\|_{\textnormal{HS}}^2=\sum_{i=n+1}^r\lambda_i(\cmatrix[\finaltouches{\bu}]).
$
\end{proof}

\;

\begin{proof}[\nopunct Proof of Lemma~\ref{lemma:bilipschitz}]
We first prove that if $\act$ is bilipschitz, then it is bijective. We have $|\act(x)-\act(y)|\ge\eta|x-y|>0$ whenever $x\neq y$, meaning that $\act$ is injective. Additionally, by setting $y=0$ and letting $x\to\pm\infty$, it is clear that $|\act(x)|\to\infty.$ Since $\act$ is continuous and injective, it is straightforward to conclude that either $\act(\pm\infty)=\pm\infty$ or $\act(\pm\infty)=\mp\infty$. In both cases, $\act$ ends up being surjective.
Concerning the \review{bilipschitz property} of $\invact$, instead, notice that we can re-write Eq.~\eqref{eq:bilipschitz} as
\begin{equation*}
    \eta|\invact(\act(x))-\invact(\act(y))|\le|\act(x)-\act(y)|\le L|\invact(\act(x))-\invact(\act(y))|\quad\forall x,y\in\mathbb{R}.
\end{equation*}
As $\act$ is surjective, setting $w=\act(x)$ and $z=\act(y)$ yields $
    |\invact(w)-\invact(z)|\le\eta^{-1}|w-z|$ and $|\invact(w)-\invact(z)|\ge L^{-1}|w-z|$
for all $w,z\in\mathbb{R}$, thus concluding the proof.
\end{proof}

\;

\begin{proof}[\nopunct Proof of Proposition~\ref{prop:classes}]\;

    \begin{itemize}[before=\vspace{5pt}, itemsep=6pt, leftmargin=16pt]
        \item [{\em (i)}] Let $\autoencoder=\{(\bVE_j,\bVD_j,\bias_j, \act)\}_{j=1}^{\finaltouches{l}}\in\SBAE$. Notice that for each $j=1,\dots,n_l$ one has $e_j\circ d_j=\textsf{Id}_{n_j}.$
        In fact, given any $\mathbf{h}\in\mathbb{R}^{n_j}$,
        $$e_j(d_j(\mathbf{h}))=\act(\bVE_j((\bVD_j\invact(\mathbf{h})+\bias)-\bias))=\act(\bVE_j\bVD_j\invact(\mathbf{h}))=\act(\invact(\mathbf{h}))=\mathbf{h},$$
        due biorthogonality. Using associativity, it follows that $\mathscr{E}(\autoencoder,\mathscr{D}(\autoencoder,\mathbf{c}))=(e_l\circ\dots\circ e_1\circ d_1\circ\dots\circ d_l)(\mathbf{c})=\mathbf{c}$ for all $\mathbf{c}\in\mathbb{R}^{n_l}$.

        \item[{\em (ii)}] This is a direct consequence of (i), as it can be obtained by setting $\mathbf{c}=\mathscr{E}(\autoencoder,\mathbf{v})$ and applying $\mathscr{D}(\autoencoder, \cdot)$ to both sides of the identity.
        
        \item [{\em (iii)}] Let $\autoencoder=\{(\bV_j,\bias_j, \act)\}_{j=1}^{\finaltouches{l}}\in\SOAE$. Notice that, for every $j=1,\dots,n_l$,  the maps $\mathbf{c}\mapsto\bV_j\mathbf{c}+\bias_j$ and $\mathbf{h}\mapsto\bV_j^\top(\mathbf{h}-\bias_j)$ are both affine and 1-Lipschitz. By composition, it follows that $\lip{e_j}\le \lip{\act}$ and $\lip{d_j}\le\lip{\invact}$. Since $\lip{\mathscr{E}(\autoencoder,\cdot)}\le\prod_{j=1}^{\finaltouches{l}}\lip{e_j}$ and $\lip{\mathscr{D}(\autoencoder,\cdot)}\le\prod_{j=1}^{\finaltouches{l}}\lip{d_j}$, the conclusion follows.
    \end{itemize}
\end{proof}

\newcommand{\bx}{\mathbf{x}}
\newcommand{\bb}{\mathbf{b}}
\newcommand{\Var}{\operatorname{Var}}
\section{Complements on the He initialization}
\label{appendix:init}
Hereby, we derive an extension of the so-called \emph{He initialization} by generalizing the construction proposed in \cite{he2015delvingdeeprectifierssurpassing} to more complex bilipschitz activations. To this end, let $f$ be a \review{bilipschitz} function with $f(0) = 0$. 
The generic forward pass of a layer mapping $\mathbb{R}^{\review{n}}\to\mathbb{R}^q$ reads $\bx = \bW f(\bz) + \bb$, with $\bx, \bb \in \mathbb{R}^q$, $\bW \in \mathbb{R}^{q \times \review{n}}$ and $\bz \in \mathbb{R}^{\review{n}}$. 
Following \cite{he2015delvingdeeprectifierssurpassing}, we shall set $\mathbf{b}\equiv\mathbf{0}$ and assume that
\begin{itemize}[itemsep=2pt,leftmargin=20pt,after=\vspace{6pt},before=\vspace{6pt}]
    \item[\emph{(i)}] $\bz = [z_1,\dots,z_{\review{n}}]^\top$ is a random vector with i.i.d. components $z_j \overset{\text{i.i.d.}}{\sim} \omega$;
    \item[\emph{(ii)}] $\omega$ is \review{centered} and has finite second moment;
    \item[\emph{(iii)}] $\bW = (w_{ij})_{i,j}$ is drawn at random according to $w_{ij} \overset{\text{i.i.d.}}{\sim} \mathcal{N}(0,\sigma^2)$. 
\end{itemize}
The idea is to pick $\sigma$ such that the variance of $\mathbf{z}$ transfers to $\mathbf{x}$ in a controlled manner. In this respect, we notice that if $\mathbf{x}=[x_1,\dots,x_q]^\top$ then, for all $i=1,\dots,q$,
\begin{equation}
    \label{eq: var xi}
    \Var(x_i) = \Var\bigg(\sum_{j=1}^{\review{n}} w_{ij}f(z_j)\bigg) = \sum_{j=1}^{\review{n}}\Var(w_{ij}f(z_j)) = n\sigma^2\mathbb{E}|f(z_1)|^2,
\end{equation}
due independence and identical distribution. Indeed, 
\begin{multline*}
\Var(w_{ij}f(z_j)) =\mathbb{E}[w_{ij}^2f(z_j)^2]-\mathbb{E}[w_{ij}f(z_j)]^2=\\=\mathbb{E}|w_{ij}|^2\mathbb{E}|f(z_j)|^2-\mathbb{E}[w_{ij}]^2\mathbb{E}[f(z_j)]^2=\sigma^2\mathbb{E}|f(z_j)|^2,
\end{multline*}since $\mathbb{E}[w_{i,j}]=0$. Thus, we are left to study
    \review{$\mathbb{E}|f(z_1)|^2 = \int_{-\infty}^{+\infty} |f(z)|^2\omega(z)dz,$}
which can be bounded from above and below, using the stability properties of $f$. Specifically, using 
the fact that $f(0)=0$, 
\review{and $\lip{f^{-1}}^{-1}|z|\le|f(z)|\le \lip{f}|z|$ for all $z\in\mathbb{R},$ one has}
\begin{equation*}
    \review{\Var(z_1)\lip{f^{-1}}^{-2} \le \int_{-\infty}^{+\infty} |f(z)|^2\omega(z)dz  \le\Var(z_1) \lip{f}^{2}.} 
\end{equation*}
Combining the above with \eqref{eq: var xi} yields
\begin{equation}
    \label{eq:sigma-bounds}
    \review{n\sigma^2\Var(z_1)\lip{f^{-1}}^{-2}} \le \Var(x_i) \le \review{n\sigma^2\Var(z_1) \lip{f}^{2}}. 
\end{equation}
In the wake of \cite{he2015delvingdeeprectifierssurpassing}, we select $\sigma$ such that $\Var(x_i)\approx\Var(z_1)$. Heuristically, we achieve so by taking the harmonic mean between the lower and upper bound values of $\sigma$ derived from \eqref{eq:sigma-bounds}, which results in $\sigma^{\review{2}} = \review{2[n( \lip{f^{-1}}^{-2} + \lip{f}^2)]^{-1}}$. 

Note that, in the case of the LeakyReLU, this choice coincides exactly with the initialization proposed in \cite{he2015delvingdeeprectifierssurpassing}.

\section{\review{Complements on biorthogonal matrices}}
\label{appendix:biorth}

\review{Let $q,r\in\mathbb{N}_+$ with $r\le q$. Set $d:=\min\{r,q-r\}$. As discussed in Section~\ref{sec:symmetric}, we can parametrize a pair of biorthogonal matrices $\mathbf{\Xi},\mathbf{\Theta}\in\mathbb{R}^{q\times r}$ using a list of five matrices $\mathbf{X},\mathbf{U},\mathbf{S},\mathbf{Z},\mathbf{Q}$ where: $\mathbf{X}\in\mathbb{R}^{q\times (r+d)}$ is orthonormal; $\mathbf{U},\mathbf{Z}\in\mathbb{R}^{r\times r}$ are orthonormal; $\mathbf{S}\in\mathbb{R}^{r\times r}$ is diagonal with strictly positive entries; $\mathbf{Q}\in\mathbb{R}^{d\times r}$ is arbitrary.
We recall that such parametrization is given by
$$\mathbf{\Xi}:=\mathbf{X}\left[\begin{array}{c}
     \mathbf{U}\mathbf{S}\mathbf{Z}^\top\\
     \mathbf{0}
\end{array}\right]\quad\quad\text{and}\quad\quad\mathbf{\Theta}:=\mathbf{X}\left[\begin{array}{c}
     \mathbf{U}\mathbf{S}^{-1}\mathbf{Z}^\top\\
     \mathbf{Q}
\end{array}\right].$$
To see that the biorthogonality constraint is fullfilled, simply notice that
\begin{align*}
\mathbf{\Xi}^\top\mathbf{\Theta}&=[\mathbf{Z}\mathbf{S}\mathbf{U}^\top\;\;
     \mathbf{0}]\;\mathbf{X}^\top\mathbf{X}\left[\begin{array}{c}
     \mathbf{U}\mathbf{S}^{-1}\mathbf{Z}^\top\\
     \mathbf{Q}
\end{array}\right]=[\mathbf{Z}\mathbf{S}\mathbf{U}^\top\;\;
     \mathbf{0}]\left[\begin{array}{c}
     \mathbf{U}\mathbf{S}^{-1}\mathbf{Z}^\top\\
     \mathbf{Q}
\end{array}\right]=\\&=\mathbf{Z}\mathbf{S}\mathbf{U}^\top\mathbf{U}\mathbf{S}^{-1}\mathbf{Z}^\top=\mathbf{I}_r.
\end{align*}
In fact, since $\mathbf{Z}$ is a square matrix, one has $\mathbf{Z}^\top\mathbf{Z}=\mathbf{I}_r\implies \mathbf{Z}\mathbf{Z}^\top=\mathbf{I}_r$.}

\review{Such parametrization is surjective, meaning that any pair of biorthogonal matrices can be written in such a way. To see this, assume $\mathbf{\Xi}$ and $\mathbf{\Theta}$ are given. Let $\mathbf{X}_1\in\mathbb{R}^{q\times r}$ be an orthonormal matrix spanning the same $r$-dimensional subspace of $\mathbf{\Xi}$. Construct $\mathbf{X}\in\mathbb{R}^{q\times (r+d)}$ by adding $d$ columns to $\mathbf{X}_1$ such that $\mathbf{X}$ is orthonormal and
$$\text{span}(\mathbf{X})\supseteq \text{span}(\mathbf{\Xi})\cup\text{span}(\mathbf{\Theta}).$$
Evidently, this is always possible. 
Indeed, if $d=r$, we simply notice that $$\dim(\text{span}(\mathbf{\Xi})\cup\text{span}(\mathbf{\Theta}))\le r+r=r+d.$$ Conversely, if $d=q-r$, then $d+r=q$, meaning that $\text{span}(\mathbf{X})=\mathbb{R}^q$; the inclusion is then trivial as $\text{span}(\mathbf{\Xi}),\text{span}(\mathbf{\Theta})\subseteq\mathbb{R}^q$.
With such construction, one then has
$$\mathbf{\Xi}:=\mathbf{X}\left[\begin{array}{c}
     \mathbf{B}\\
     \mathbf{0}
\end{array}\right]\quad\quad\text{and}\quad\quad\mathbf{\Theta}:=\mathbf{X}\left[\begin{array}{c}
     \mathbf{C}\\
     \mathbf{Q}
\end{array}\right],$$
for some invertible matrices $\mathbf{B},\mathbf{C}\in\mathbb{R}^{r\times r}$ and $\mathbf{Q}\in\mathbb{R}^{d\times r}$. Since $\mathbf{I}_r=\mathbf{\Xi}^\top\mathbf{\Theta}=\mathbf{B}^\top\mathbf{C}$, it follows that $\mathbf{C}=(\mathbf{B}^\top)^{-1}.$ Introducing an SVD decomposition for $\mathbf{B}$ yields the desired representation.}

\review{As a final note, we mention that this parametrization is redundant, as the manifold of biorthogonal matrices is $2qr-r^2$-dimensional \cite{otto2023learning}, whereas the proposed parametrization features 
$$q(r+d)-\frac{1}{2}(r+d)(r+d+1)+r^2-r+r+dr = \frac{1}{2}(r+d)(2q+r-d-1)$$
parameters. In particular, if $d=r$, then our approach entails a total of $r(r-1)\ge0$ ``extra'' parameters. Conversely, the latter are $(\frac{q}{2}-r)^2+\frac{q}{2}(\frac{q}{2}-1)\ge0$ whenever $d=q-r$.}

\section{Supplementary material}

\subsection{Complementary results on the EYS initialization}
\label{supp_mat:init}
The following Fig.~\ref{fig:init-study-sm} refer\review{s} to the study conducted on the \textsf{PGA} test case concerning the performances of the EYS initialization \textit{before training}. Results refer to different architectures and activations of varying sharpness, \review{as indicated below; other specifics are the same as the ones of Fig. 5 in the main text.}

\begin{figure}[h!]
    \centering
    \includegraphics[width=0.83\linewidth]{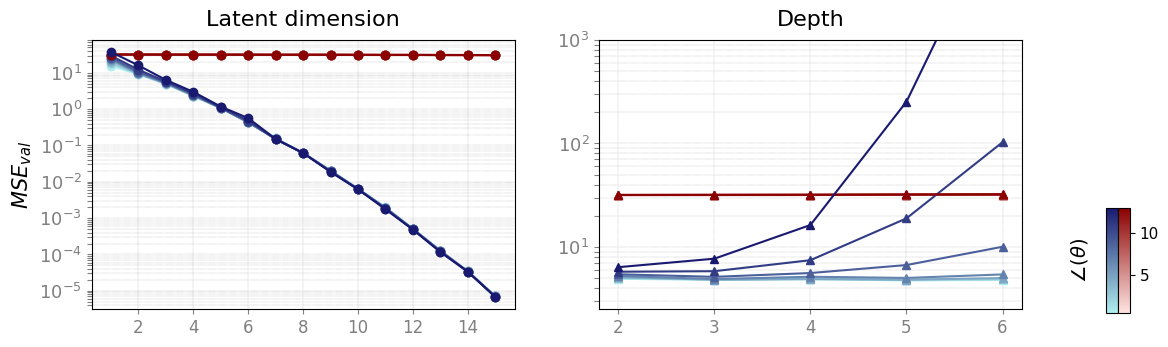}
    \includegraphics[width=0.83\linewidth]{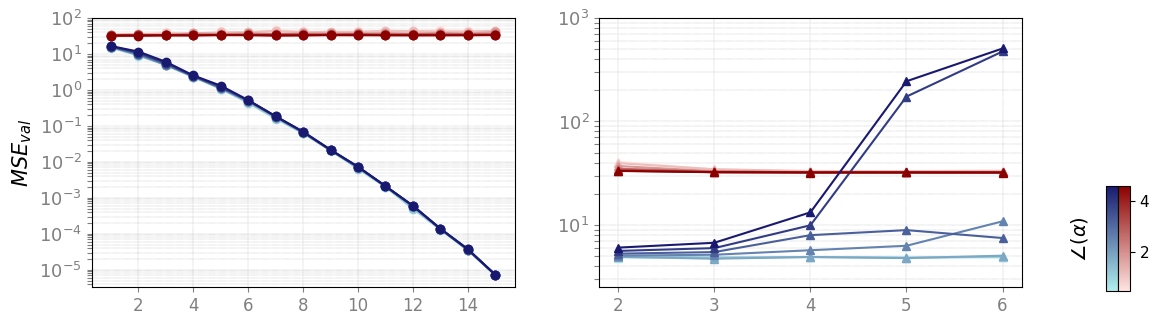}
    \includegraphics[width=0.83\linewidth]{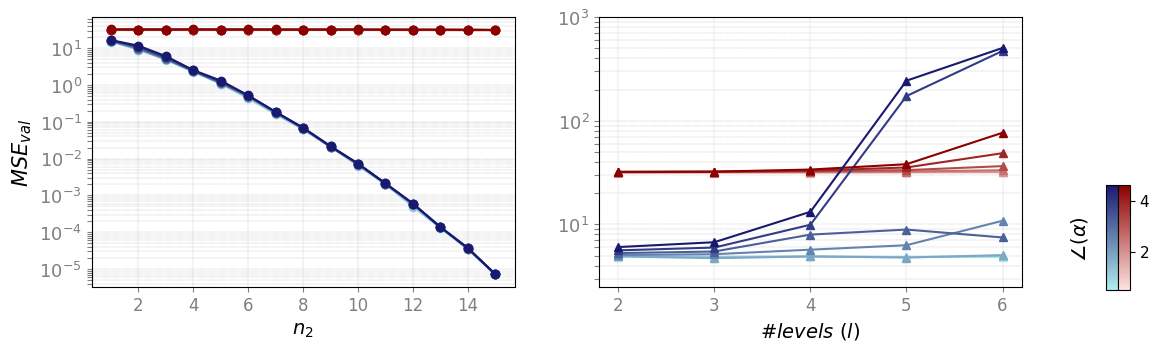}
    \caption{\small Comparison between the EYS {\em ({\color{darkblue}\textbf{blue}})} and the standard random initialization {\em ({\color{darkred}\textbf{red}})} for the {\em\textsf{PGA}} case study. \review{{\em Top:} {\em \textsf{SOAE}}s with $HypAct_{\theta}$. {\em Middle:} {\em \textsf{SAE}}s with $LeakyReLU_{\alpha,5/4}$. {\em Bottom:} {\em \textsf{SOAE}}s} with $LeakyReLU_{\alpha,5/4}$\finaltouches{.} }
    \label{fig:init-study-sm}
\end{figure}

Results for \textsf{SBAE} and \textsf{SOAE} are reported together as, for these architectures, both the EYS initialization and the literature standard coincide. We recall, in fact, that we considered the approach proposed in \cite{otto2023learning} as our benchmark: there, the authors propose to initialize \textsf{SBAE} architectures with $\bVE_j=\bVD_j^\top$. Consequently, right after initialization, the \textsf{SBAE} module is equivalent to a \textsf{SOAE}. This is also true in terms of the underlying parametrization, as we use a procedure analogous to the one discussed in Section 4.2 {\em (ii)} to initialize the matrices $\tilde{\mathbf{X}}_j,\tilde{\mathbf{Y}}_j,\tilde{\mathbf{Z}}_j,\mathbf{Q}_j$ and the vector $\mathbf{s}_j$.

\subsection{Complementary results on the comparative study}
\label{supp_mat:comparative_study}
We report below the complementary results for the comparative study \review{presented in Section 5}. Specifically, Tables~\ref{tab:pga}-\ref{tab:els} read as Table 3 in the paper but refer to the \textsf{PGA} and \textsf{ELS} case studies, respectively. Conversely, Figures~\ref{fig:comparison_sm1},~\ref{fig:comparison_sm2}, and~\ref{fig:comparison_sm3}  are analogous to Fig. 6 in the paper but discuss different choices of \review{$\act$} and its sharpness value.

\begin{table}[h!]
\centering
\begin{footnotesize}
\renewcommand{\arraystretch}{1.05}
\begin{tabular}{lllll}
\toprule
 &  &  & \textbf{MSE} & \textbf{MRE} \\
\midrule
\multirow[c]{8}{*}{$\text{HypAct}_{\theta}$} & \multirow[c]{4}{*}{$\angle(\theta) = 0.5$} & $\textsf{AE}$ & $(7.55 \pm 1.90) \times 10^{-3}$ & $(1.30 \pm 0.07) \times 10^{-2}$ \\
 &  & $\textsf{SAE}$ & $(7.13 \pm 1.23) \times 10^{-3}$ & $(1.33 \pm 0.06) \times 10^{-2}$ \\
 &  & $\textsf{SBAE}$ & $(1.21 \pm 0.20) \times 10^{-2}$ & $(1.75 \pm 0.07) \times 10^{-2}$ \\
 &  & $\textsf{SOAE}$ & $(3.27 \pm 0.23) \times 10^{0}$ & $(3.02 \pm 0.09) \times 10^{-1}$ \\
\cline{2-5}
 & \multirow[c]{4}{*}{$\angle(\theta) = 3.0$} & $\textsf{AE}$ & $(1.81 \pm 0.41) \times 10^{-3}$ & $(6.53 \pm 0.32) \times 10^{-3}$ \\
 &  & $\textsf{SAE}$ & $(1.79 \pm 0.35) \times 10^{-3}$ & $(6.60 \pm 0.31) \times 10^{-3}$ \\
 &  & $\textsf{SBAE}$ & $(5.23 \pm 0.92) \times 10^{-3}$ & $(1.17 \pm 0.05) \times 10^{-2}$ \\
 &  & $\textsf{SOAE}$ & $(1.57 \pm 0.28) \times 10^{-2}$ & $(1.99 \pm 0.09) \times 10^{-2}$ \\
\cline{1-5} 
\multirow[c]{8}{*}{$\text{LeakyReLU}_{\alpha,5/4}$} & \multirow[c]{4}{*}{$\angle(\alpha) = 0.5$} & $\textsf{AE}$ & $(7.64 \pm 2.62) \times 10^{-3}$ & $(1.22 \pm 0.08) \times 10^{-2}$ \\
 &  & $\textsf{SAE}$ & $(5.84 \pm 1.56) \times 10^{-3}$ & $(1.14 \pm 0.06) \times 10^{-2}$ \\
 &  & $\textsf{SBAE}$ & $(7.59 \pm 1.81) \times 10^{-3}$ & $(1.35 \pm 0.07) \times 10^{-2}$ \\
 &  & $\textsf{SOAE}$ & $(2.78 \pm 0.19) \times 10^{0}$ & $(2.79 \pm 0.08) \times 10^{-1}$ \\
\cline{2-5}
 & \multirow[c]{4}{*}{$\angle(\alpha) = 3.0$} & $\textsf{AE}$ & $(4.42 \pm 0.75) \times 10^{-3}$ & $(1.05 \pm 0.05) \times 10^{-2}$ \\
 &  & $\textsf{SAE}$ & $(2.67 \pm 0.36) \times 10^{-3}$ & $(8.18 \pm 0.35) \times 10^{-3}$ \\
 &  & $\textsf{SBAE}$ & $(7.43 \pm 0.88) \times 10^{-3}$ & $(1.42 \pm 0.05) \times 10^{-2}$ \\
 &  & $\textsf{SOAE}$ & $(2.09 \pm 0.26) \times 10^{-2}$ & $(2.35 \pm 0.09) \times 10^{-2}$ \\
\bottomrule
\end{tabular}

\renewcommand{\arraystretch}{1}
\end{footnotesize}
\vspace{6pt}
\caption{\small 
{\em \textsf{PGA}} test case\review{, with EYS initialization (see Section~\ref{supp_mat:comparative_study}}).} 
\label{tab:pga}
\vspace{-1cm}
\end{table}

\begin{table}[h!]
\centering
\begin{footnotesize}
\renewcommand{\arraystretch}{1.05}
\begin{tabular}{lllll}
\toprule
 &  &  & \textbf{MSE} & \textbf{MRE} \\
\midrule
\multirow[c]{8}{*}{$\text{HypAct}_{\theta}$} & \multirow[c]{4}{*}{$\angle(\theta) = 0.5$} & $\textsf{AE}$ & $(7.86 \pm 0.48) \times 10^{-3}$ & $(1.42 \pm 0.06) \times 10^{-2}$ \\
 &  & $\textsf{SAE}$ & $(2.67 \pm 0.21) \times 10^{-3}$ & $(7.93 \pm 0.25) \times 10^{-3}$ \\
 &  & $\textsf{SBAE}$ & $(3.10 \pm 0.28) \times 10^{-3}$ & $(8.26 \pm 0.29) \times 10^{-3}$ \\
 &  & $\textsf{SOAE}$ & $(2.62 \pm 0.22) \times 10^{-3}$ & $(7.70 \pm 0.26) \times 10^{-3}$ \\
\cline{2-5}
 & \multirow[c]{4}{*}{$\angle(\theta) = 3.0$} & $\textsf{AE}$ & $(1.24 \pm 0.47) \times 10^{-2}$ & $(1.46 \pm 0.06) \times 10^{-2}$ \\
 &  & $\textsf{SAE}$ & $(2.95 \pm 0.23) \times 10^{-3}$ & $(8.15 \pm 0.27) \times 10^{-3}$ \\
 &  & $\textsf{SBAE}$ & $(2.00 \pm 0.22) \times 10^{-3}$ & $(6.44 \pm 0.23) \times 10^{-3}$ \\
 &  & $\textsf{SOAE}$ & $(2.32 \pm 0.21) \times 10^{-3}$ & $(7.10 \pm 0.25) \times 10^{-3}$ \\
\cline{1-5} 
\multirow[c]{8}{*}{$\text{LeakyReLU}_{\alpha,5/4}$} & \multirow[c]{4}{*}{$\angle(\alpha) = 0.5$} & $\textsf{AE}$ & $(2.83 \pm 0.20) \times 10^{-3}$ & $(8.18 \pm 0.33) \times 10^{-3}$ \\
 &  & $\textsf{SAE}$ & $(3.48 \pm 0.25) \times 10^{-3}$ & $(8.97 \pm 0.26) \times 10^{-3}$ \\
 &  & $\textsf{SBAE}$ & $(8.80 \pm 0.65) \times 10^{-4}$ & $(4.61 \pm 0.16) \times 10^{-3}$ \\
 &  & $\textsf{SOAE}$ & $(1.36 \pm 0.10) \times 10^{-3}$ & $(5.87 \pm 0.26) \times 10^{-3}$ \\
\cline{2-5}
 & \multirow[c]{4}{*}{$\angle(\alpha) = 3.0$} & $\textsf{AE}$ & $(1.89 \pm 0.56) \times 10^{-2}$ & $(1.78 \pm 0.08) \times 10^{-2}$ \\
 &  & $\textsf{SAE}$ & $(5.43 \pm 0.42) \times 10^{-3}$ & $(1.11 \pm 0.03) \times 10^{-2}$ \\
 &  & $\textsf{SBAE}$ & $(9.03 \pm 1.10) \times 10^{-4}$ & $(4.38 \pm 0.16) \times 10^{-3}$ \\
 &  & $\textsf{SOAE}$ & $(2.16 \pm 0.19) \times 10^{-3}$ & $(6.89 \pm 0.23) \times 10^{-3}$ \\
\bottomrule
\end{tabular}

\renewcommand{\arraystretch}{1}
\end{footnotesize}
\vspace{6pt}
\caption{\small 
{\em \textsf{ELS}} test case\review{, with EYS initialization (see Section~\ref{supp_mat:comparative_study}}).}
\label{tab:els}
\end{table}

\begin{figure}[h!]
    \centering
    \includegraphics[width=0.99\linewidth]{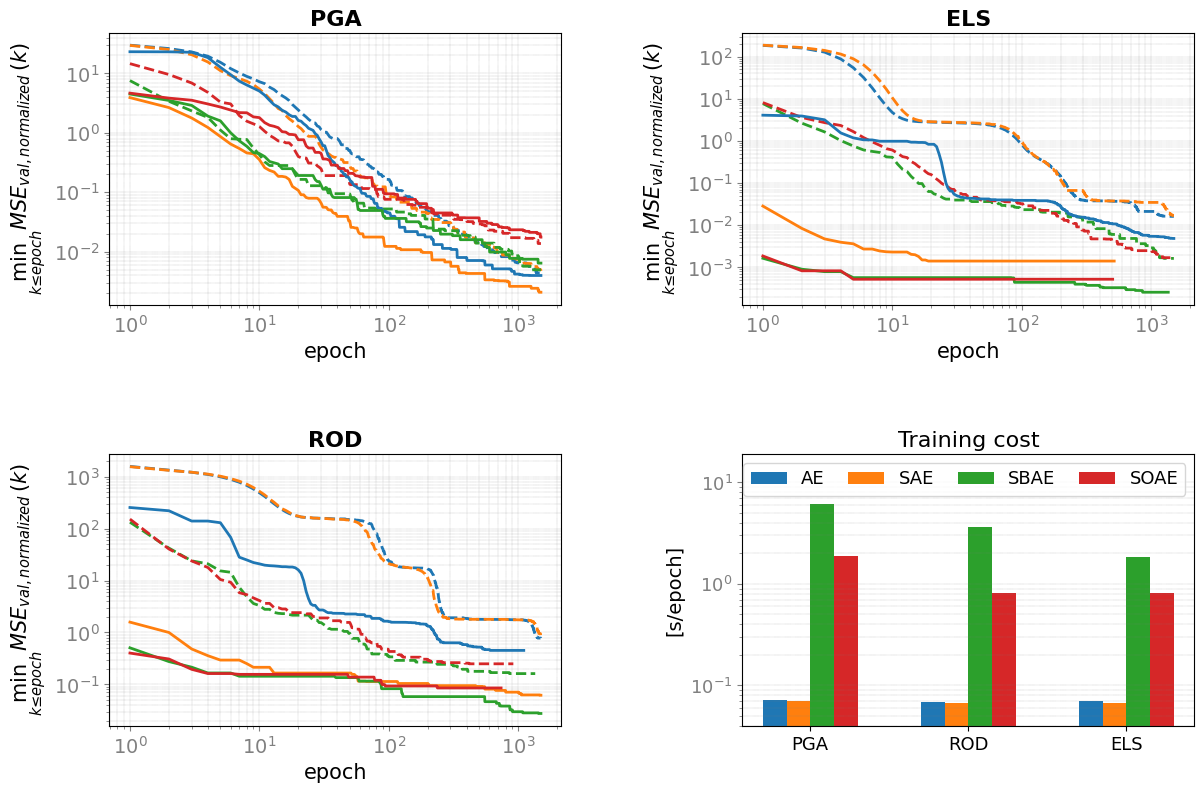}
    \caption{\small Visualization of the training dynamics, with $LeakyReLU_{\alpha,5/4}$, with $\alpha$ such that $\angle(\alpha)=3.0$.}
    \label{fig:comparison_sm1}
\end{figure}

\begin{figure}[h!]
    \centering
    \includegraphics[width=0.99\linewidth]{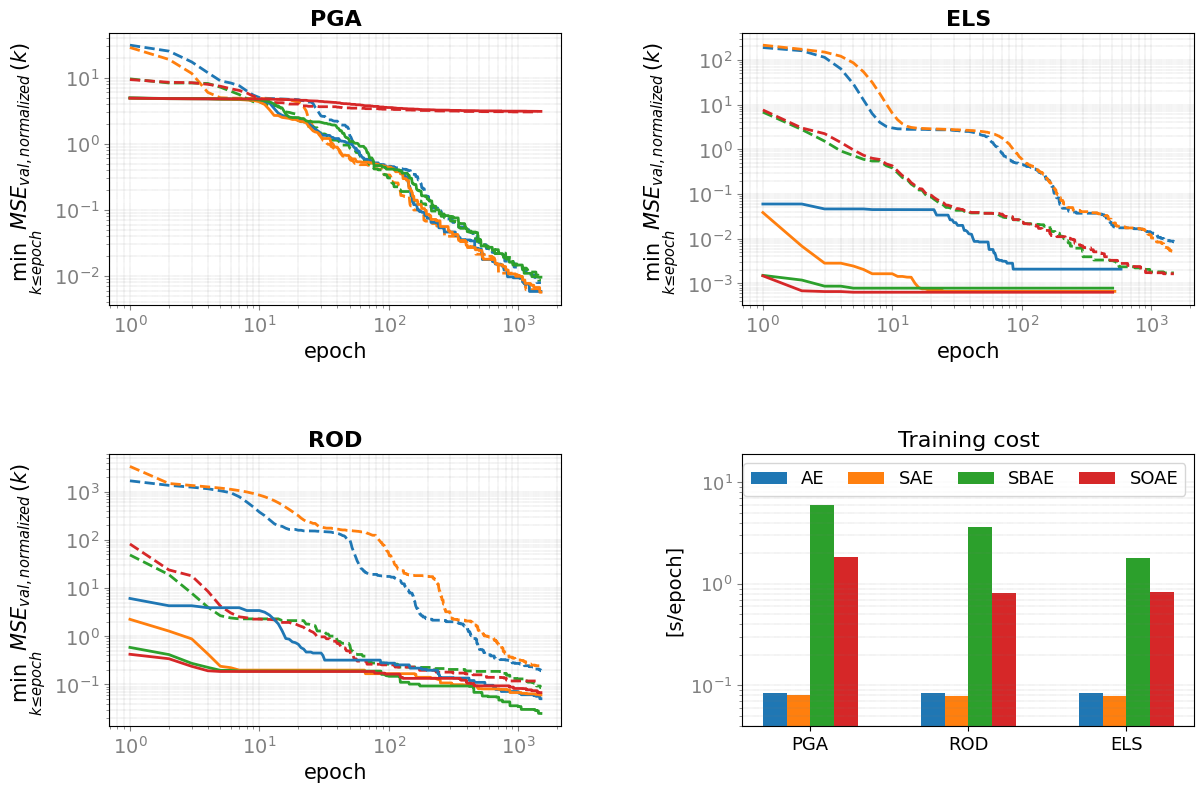}
    \caption{\small Visualization of the training dynamics, with $HypAct(\theta)$, with $\theta$ such that $\angle(\theta)=0.5$.}
    \label{fig:comparison_sm2}
\end{figure}

\begin{figure}[h!]
    \centering
    \includegraphics[width=0.99\linewidth]{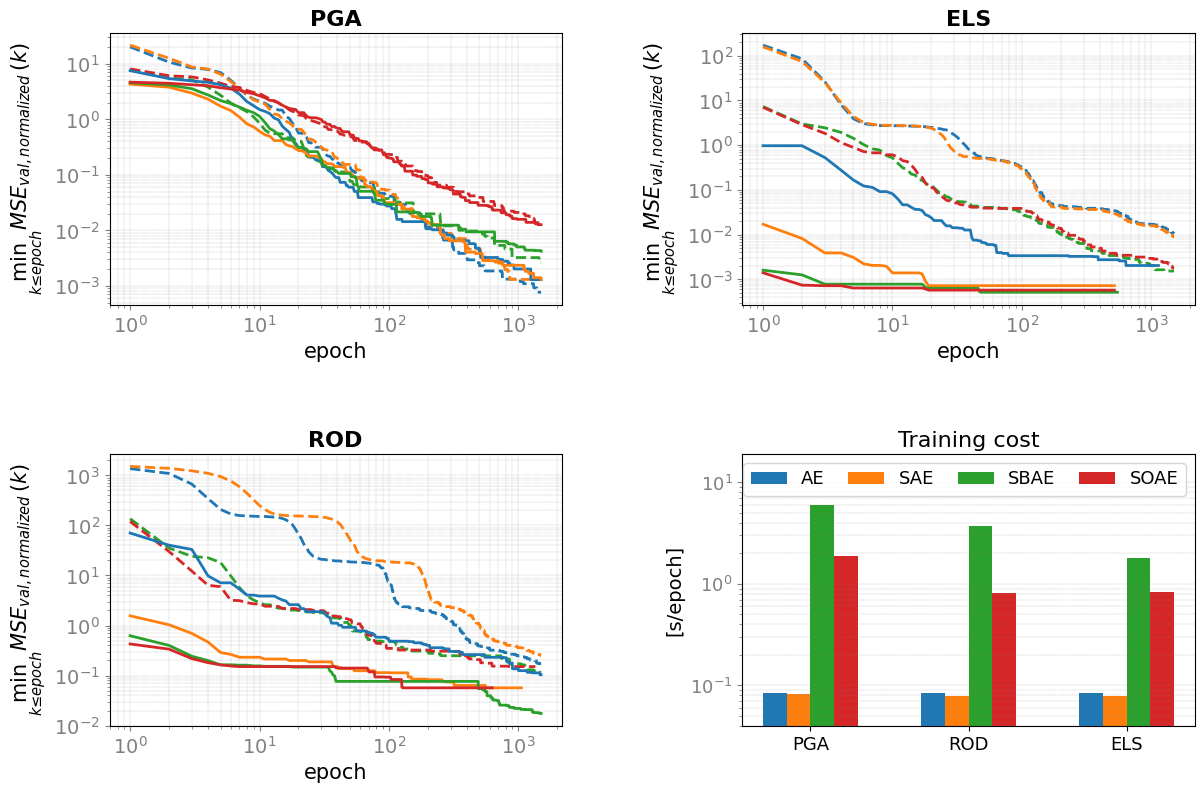}
    \caption{\small Visualization of the training dynamics, with $HypAct(\theta)$, with $\theta$ such that $\angle(\theta)=3.0$.}
    \label{fig:comparison_sm3}
\end{figure}

\review{
\subsection{Exploration of the hyperparameters' space}
To complement the analysis reported in Section 5 of the main text, we report below some additional results concerning the impact of network hyperparameters on the reconstruction accuracy of EYS-initialized and standard-initialized \emph{deep symmetric autoencoders}. These include design hyperparameters---number of layers, neurons, etc.---and training hyperparameters---learning rate, decay schedulers, etc.
%
To limit the otherwise excessive computational demand, this analysis was limited to SAE architectures, which are by far the most efficient to train.
%
As anticipated, aside from the initialization procedure, in our hyperparameter sweep we vary:
\begin{itemize}[itemsep=2pt,leftmargin=*,after=\vspace{2pt},before=\vspace{2pt}]
    \item[--] the activation function in $
    \{\textnormal{LeakyReLU}_{\alpha,5}, \textnormal{HypAct}_{\theta}\}$ with $\alpha,\theta$ such that $\angle(\cdot) \in \{0.5,3\}$;
    \item[--] the learning rate, $lr \in \{3\times10^{-4}, 6 \times10^{-4}, 1\times10^{-3}\}$;
    \item[--] the learning rate scheduler in $\{\mathtt{None},\; \mathtt{CosineAnnealing},\; \mathtt{OneCycle}\},$ with the \finaltouches{\texttt{OneCycle}} scheduler featuring a maximum learning rate equal to $10^{-3}.$
    \item[--] the architecture skeletons, according to problem-dependent set of choices, see Table~\ref{tab:hypopt_skeletons}.
\end{itemize}
 \begin{table}
        \centering
        \footnotesize{
        \begin{tabular}{c  c  c}
            \toprule
            \textsf{PGA} & \textsf{ROD} & \textsf{ELS} \\
            \midrule
            $[514, 64, 3]$ & $[4347, 15, 5]$& $[882, 8, 5]$\\
            $[514, 64, 15, 3]$ & $[4347, 15, 7, 5]$ & $[882, 15, 10, 8]$\\ 
            $[514, 64, 15, 7, 3]$ & $[4347, 20, 7]$& $[882, 15, 10]$\\
            $[514, 32, 15, 3]$ & $[4347, 30, 20, 10, 5]$ & $[882, 10, 8]$\\
            $[514, 16, 10, 3]$ & $[4347, 30, 10, 5]$& $[882, 30, 15, 10, 8]$\\
            $[514, 64, 15, 5]$ & $[4347, 20, 10, 7]$ & $ [882, 30, 10]$\\
            $[514, 64, 18, 7]$ & $[4347, 30, 20, 15, 10, 5]$ & $[882, 15, 9, 6]$\\ 
            \bottomrule
        \end{tabular}
        }
        \vspace{6pt}
        \caption{\small Architecture skeletons for the hyperparameter sweep of Section~\ref{supp_mat:sample_size}.}
        \label{tab:hypopt_skeletons}
    \end{table}

\begin{figure}
    \centering
    \includegraphics[width=0.99\linewidth]{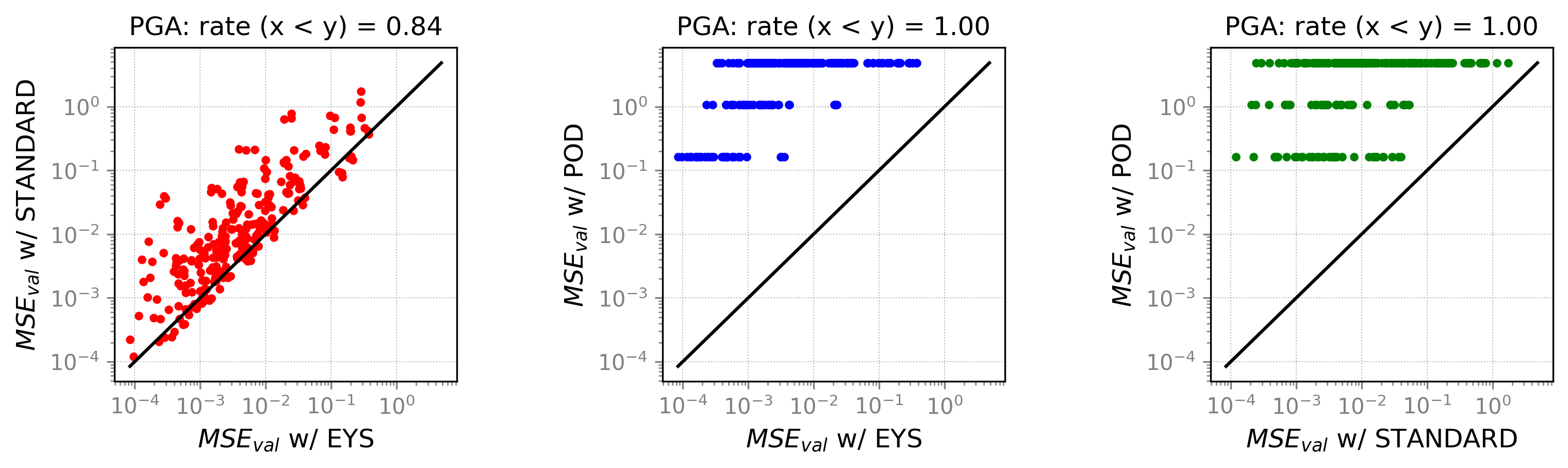} \\
    \includegraphics[width=0.99\linewidth]{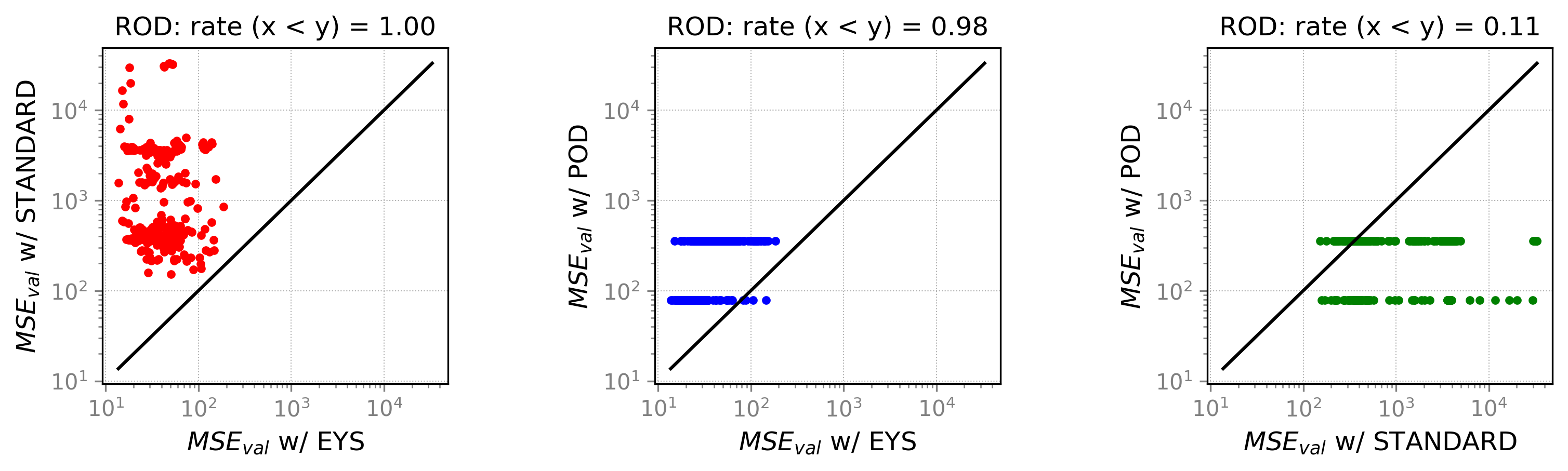} \\
    \includegraphics[width=0.99\linewidth]{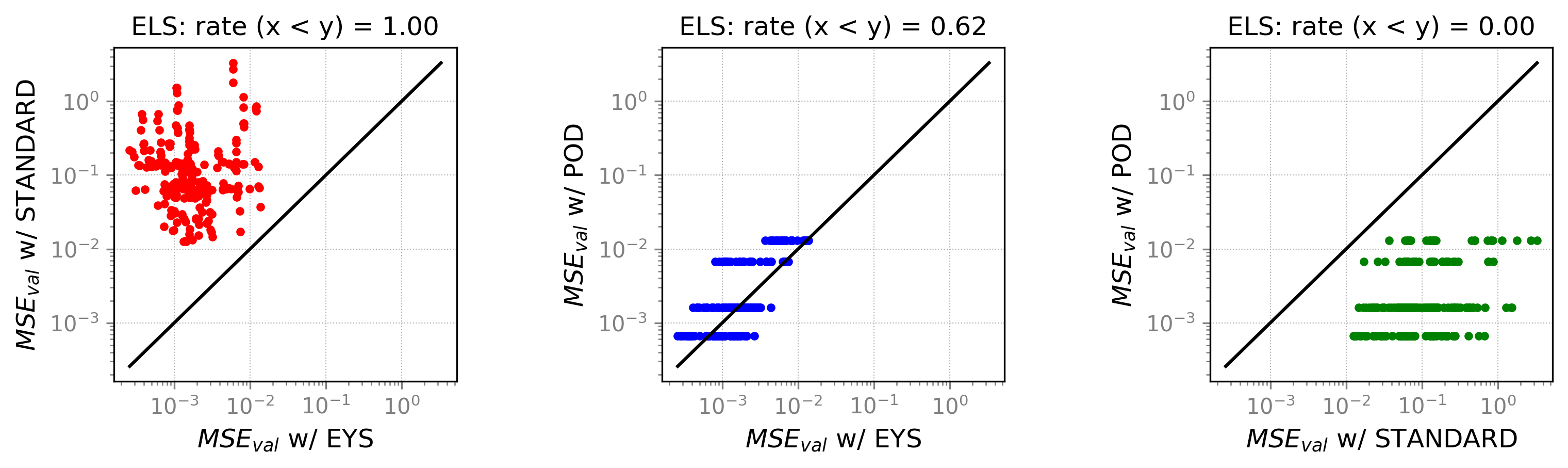} \\
    \caption{\small For any choice of the design and training hyperparameters, we compare the validation MSE of the best (among 3 runs) EYS- and standard- initialized architectures after training (first column). The POD baseline having the same latent dimension as the trained \em{\textsf{SAE}} is reported for reference (second and third columns).}
    \label{fig:hypopt}
\end{figure}
We emphasize that each run is repeated 3 times to increase the robustness of the analysis. \finaltouches{Our findings are summarized in~\ref{fig:hypopt}, where we report and compare the validation performance of all trained architectures.}
The hyperparameters’ space exploration shows that the EYS initialization
method consistently outperforms its standard counterpart (cf. Fig.~\ref{fig:hypopt}, left column),
quite independently of the choice of the design and training hyperparameters, thereby
demonstrating the robustness of the proposed initialization strategy.
\finaltouches{At the same time, there also clear differences in between the three test cases: while EYS is basically always advantageous when considering the ROD and ELS examples, the improvement rate drops to 84\% in the PGA test case. Still, we notice that this is consistent with our original findings, which suggested that the advantages of the EYS initialization seem to be tied to the linear reducibility of the problem at hand: cf. \S5.2 in the main text.}}

\review{
\subsection{Sensitivity analysis with respect to the training sample size}
\label{supp_mat:sample_size}
Within this Section, we aim to further study the robustness of our analysis with respect to varying training-validation sample sizes. To this end, our data generation and training setup is set as follows,
\begin{itemize}[itemsep=2pt, leftmargin=*,after=\vspace{2pt},before=\vspace{2pt}]
    \item[--] we select different sample sizes, $\finaltouches{S_{\text{train+val}}} \in \{50, 150, 450, 1350\}$, for model training and validation. 
    The training-validation split is set to 80:20. \finaltouches{As usual, we use the training data for the computation and minimization the loss function, whereas we rely on the validation set for checkpoint selection and early stopping;}
    \item[--] we perform 5 independent runs, obtained by re-generating the training and the validation data (each run being done independently), for all architectures and all sample sizes mentioned above;
    \item[--] we employ a fixed test set consisting of 500 independent simulations.
\end{itemize}
\begin{figure}
    \centering
    \includegraphics[width=\linewidth]{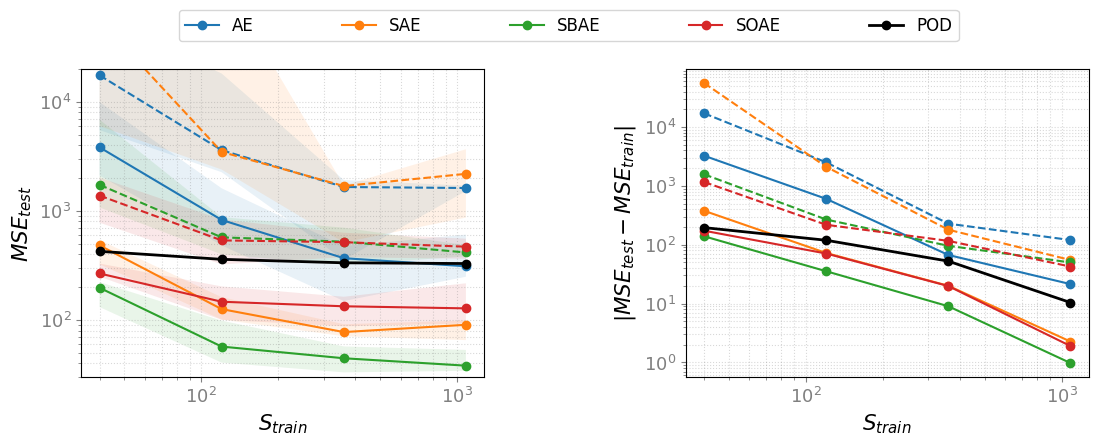}
    \caption{\small We report the results of the analysis of the sample size sensitivity study, using continuous and dashed lines to indicate architectures initialized with EYS and the standard strategy, respectively. The POD baseline is added for reference. We report the median value (among the 5 different runs), along with min-max bands (whenever they do not degrade the visualization).}
    \label{fig:sample_size_robustness}
\end{figure}
To limit the otherwise excessive computational burden of the proposed analysis, for this numerical experiment, we consider only the \textsf{ROD} test case, using $\textnormal{LeakyReLU}_{5,5/4}$ as activation function. We emphasize that the other hyperparameters coincide with the ones reported in Numerical Experiments Section of the main text. We showcase the results of our analysis in Fig.~\ref{fig:sample_size_robustness}, which shows that: 
\begin{itemize}[itemsep=2pt,leftmargin=*,after=\vspace{2pt},before=\vspace{2pt}]
    \item[\em (i)] when initialized with the EYS strategy, all symmetric autoencoders show a mild decay of the test error $MSE_{test}$, suggesting that properly initialized symmetric autoencoders can readily benefit from an increasing amount of data samples;
    \item[\em (ii)] the test error of standard-initialized symmetric autoencoders seems to saturate earlier and with a larger saturation value, when compared to their EYS-initialized counterparts, corroborating again the importance of model initialization; 
    \item[\em (iii)] especially when the number of samples is quite small, the EYS-initialized architectures outperform their standard-initialized counterparts in terms of \finaltouches{generalization gap}, here defined as $|MSE_{train} - MSE_{test}|$, thereby hinting that proper initialization strategies can be effective in enhancing generalization performance of symmetric autoencoders.
\end{itemize}
}

\printbibliography
\end{document}